\documentclass[11pt]{amsart}

\usepackage{amsmath}
\usepackage{amssymb}
\usepackage{amsfonts}
\usepackage{amsthm}  
\usepackage{latexsym}
\usepackage{graphicx}    
\usepackage{tikz} 
\usepackage{tikz-cd} 
\usepackage{color} 
\usepackage{mathdots}
\usepackage{todonotes} 
\usepackage{mathrsfs} 
\usepackage{dutchcal} 
\usepackage{verbatim} 
\usepackage{pdfpages} 
\usepackage{mathtools} 
\usepackage[pdftex,linktocpage=true,pdfborder=true]{hyperref}
\usepackage[nameinlink]{cleveref} 

\crefname{thm}{theorem}{theorems}
\crefname{lem}{lemma}{lemmas}
\crefname{cor}{corollary}{corollaries}
\crefname{prop}{proposition}{propositions}
\crefname{defn}{definition}{definitions}
\crefname{eg}{example}{examples}
\crefname{xca}{exercise}{exercises}
\crefname{conj}{conjecture}{conjectures}
\crefname{rmk}{remark}{remarks}
\crefname{qst}{question}{questions}
\crefname{obs}{observation}{observations}

\newtheorem{thm}{Theorem}[subsection]
\newtheorem*{thm*}{Theorem} 
\newtheorem{lem}[thm]{Lemma}  
\newtheorem{cor}[thm]{Corollary}
\newtheorem{prop}[thm]{Proposition}
\theoremstyle{definition}
\newtheorem{defn}[thm]{Definition}

\theoremstyle{remark}
\newtheorem{rmk}[thm]{Remark}

\newtheorem*{note}{Note} 
\numberwithin{equation}{section}

\newcommand{\ip}[2]{\langle #1 , #2 \rangle}    

\newcommand\C{\mathbb C}    
\newcommand\R{\mathbb R}    
\newcommand\N{\mathbb N}

\newcommand\style{\mathfrak}
\newcommand\GL{\mathbf{GL}} 
\newcommand\G{\mathbf{G}} 
\newcommand\Sp{\mathbf{Sp}} 
\newcommand\Hbf{\mathbf{H}} 
\newcommand\Abf{\mathbf{A}} 
\newcommand\Sbf{\mathbf{S}} 
\newcommand\Pbf{\mathbf{P}} 
\newcommand\Mbf{\mathbf{M}} 
\newcommand\Nbf{\mathbf{N}} 
\newcommand\of{\mathcal O_F} 

\DeclareMathOperator{\ran}{Im} 
\DeclareMathOperator{\spn}{span} 
\DeclareMathOperator{\Int}{Int}  
\DeclareMathOperator{\Ad}{Ad}  
\DeclareMathOperator{\supp}{Supp}
\DeclareMathOperator{\ind}{Ind} 
\DeclareMathOperator{\cind}{c-Ind} 
\DeclareMathOperator{\Hom}{Hom} 
\DeclareMathOperator{\op}{op} 

\begin{document}
\title{The support of closed orbit relative matrix coefficients} 
\author[J.~M.~Smith]{Jerrod Manford Smith}
\address{Department of Mathematics \& Statistics, University of Calgary, Calgary, AB, Canada, T2N 1N4}
\email{jerrod.smith@ucalgary.ca}
\urladdr{}
\thanks{}

\date{June 18, 2019}
\dedicatory{}
\begin{abstract}
Let $F$ be a nonarchimedean local field with odd residual characteristic and let $G$ be the $F$-points of a connected reductive group defined over $F$.  Let $\theta$ be an $F$-involution of $G$. Let $H$ be the subgroup of $\theta$-fixed points in $G$.  Let $\chi$ be a quasi-character of $H$.   A smooth complex representation $(\pi,V)$ of $G$ is $(H,\chi)$-distinguished if there exists a nonzero element $\lambda$ in $\operatorname{Hom}_H(\pi,\chi)$.  We generalize a construction of descended invariant linear forms on Jacquet modules first carried out independently by Kato and Takano (2008), and Lagier (2008) to the setting of $(H,\chi)$-distinction.   We follow the methods of Kato and Takano, providing a new proof of similar results of Delorme (2010).  Moreover, we give an $(H,\chi)$-analogue of Kato and Takano's relative version of the Jacquet Subrepresentation Theorem.  In the case that $\chi$ is unramified, $\pi$ is parabolically induced from a $\theta$-stable parabolic subgroup of $G$, and $\lambda$ arises via the closed orbit in $Q\backslash G / H$, we study the (non)vanishing of the descended forms via the support of $\lambda$-relative matrix coefficients.
\end{abstract}
\maketitle
\section{Introduction}

Let $F$ be a nonarchimedean local field with odd residual characteristic.  Let $G$ be the $F$-points of a connected reductive group $\G$ defined over $F$.  Let $\theta$ be an $F$-rational involution of $\G$ and let $H = \G^\theta(F)$ be the group of $F$-points of the $\theta$-fixed subgroup of $\G$.  We are interested in the harmonic analysis on the symmetric space $G/H$.

Let $\chi$ be a quasi-character of $H$.  
A smooth complex representation $(\pi,V)$ of $G$ is $(H,\chi)$-distinguished if there exists a nonzero element $\lambda$ in $\operatorname{Hom}_H(\pi,\chi)$.
If $\chi$ is the trivial character of $H$, then we say that $(\pi,V)$ is $H$-distinguished.
By Frobenius Reciprocity, a representation $(\pi,V)$ of $G$ is $(H,\chi)$-distinguished if and only if there is a nonzero intertwining operator from $\pi$ to the (smooth) induced representation $\ind_H^G\chi$.
Thus, it is precisely the $(H,\chi)$-distinguished representations of $G$ that may be realized in the space of smooth $\C$-valued functions on $G$ that are $\chi$-eigenfunctions for the action of $H$.
Assume that $(\pi,V)$ is $H$-distinguished.
The intertwining operator from $(\pi,V)$ to $\ind_H^G\chi$ takes a vector $v \in V$ to a relative matrix coefficient $\varphi_{\lambda,v}$ associated to the linear functional $\lambda$.
In the case that $\chi$ is trivial, the $H$-distinguished representations of $G$ are exactly those that may be realized in the space $C^\infty(G/H)$ of smooth functions on $G/H$.  In this way, the study of distinguished representations arises naturally from questions in the harmonic analysis on $G/H$.

For the moment, assume that $\chi$ is trivial.  Independently, Kato and Takano \cite{kato--takano2008} and Lagier \cite{lagier2008} studied the asymptotic properties of relative matrix coefficients via Jacquet modules.
In particular, given a nonzero element $\lambda \in \Hom_H(\pi,1)$ Kato--Takano, and Lagier constructed a canonical linear functional $\lambda_N \in \Hom_{M^\theta}(\pi_N,1)$, associated to $\lambda$, for any $\theta$-split parabolic subgroup $P$ of $G$ with $\theta$-stable Levi factor $M = P \cap \theta(P)$.
Recall that a parabolic subgroup $P$ of $G$ is $\theta$-split if and only if $\theta(P)$ is opposite to $P$.
Here we denote by $(\pi_N,V_N)$ the normalized Jacquet module of $\pi$ along $P$, where $N$ is the unipotent radical of $P$.
In addition, Kato and Takano proved a generalization of Jacquet's Subrepresentation Theorem \cite[Theorem 7.1]{kato--takano2008} (see \Cref{kt08-rsc-thm}).  Kato and Takano continued their study of relative matrix coefficients in \cite{kato--takano2010}, where they proved a generalization of Casselman's criterion for square integrability.

The results of Lagier \cite{lagier2008} were extended to the case that $\chi$ is nontrivial by Delorme \cite{delorme2010} via methods similar to Lagier's.  
Let $(\pi,V)$ be a smooth $(H,\chi)$-distinguished representation of $G$, and let $P= MN$ be a $\theta$-split parabolic subgroup of $G$ with unipotent radical $N$ and $\theta$-stable Levi $M$.
Given a nonzero $\lambda \in \Hom_H(\pi,\chi)$ Delorme canonically associates a linear functional $\lambda_{N,\chi}$ in the space $\Hom_{M^\theta}(\pi_N,\chi\vert_{M^\theta})$.
Using the linear functional $\lambda_{N,\chi}$ Delorme studies the ``constant term" of smooth $(H,\chi)$-spherical functions on $G$ \cite{delorme2010}.
In \Cref{sec-invariant-kt08}, after recalling the results of \cite{kato--takano2008,lagier2008}, we give an alternate proof of Delorme's construction of $\lambda_{N,\chi}$ for admissible $(H,\chi)$-distinguished representations of $G$ via the methods of Kato and Takano \cite{kato--takano2008}.  We study $(H,\chi)$-relatively supercuspidal representations, whose relative matrix coefficients are compactly supported in the appropriate sense, in terms of non-vanishing of the linear functionals $\lambda_{N,\chi}$ (see \Cref{thm-twisted-rsc-condition}). In addition, we extend the relative subrepresentation theorem \cite[Theorem 7.1]{kato--takano2008} to the case of $(H,\chi)$-distinction (see \Cref{thm-twisted-kt08-7-1}).

In \Cref{sec-not-rsc}, under the assumption that $\chi$ is unramified, we study the descended linear functionals on certain $(H,\chi)$-distinguished representations induced from $\theta$-stable parabolic subgroups.
Let $Q=LU$ be a $\theta$-stable parabolic subgroup of $G$ with $\theta$-stable Levi factor $L$ and unipotent radical $U$.
Let $(\rho,V_\rho)$ be an irreducible representation of $L$ and assume that $\lambda \in \Hom_{L^\theta}(\delta_Q^{1/2}\rho,\delta_{Q^\theta}\chi\vert_{L^\theta})$ is nonzero.
Let $\iota_Q^G\rho$ be the normalized parabolically induced representation of $G$ obtained from $(\rho,V_\rho)$.   
  It is well known (see \Cref{lem-hom-injects}) that one may construct a nonzero linear functional $\lambda^G \in \Hom(\iota_Q^G\rho,\chi)$ from $\lambda$ via the Mackey theory for $Q\backslash G / H$.  
 We provide a necessary condition for $\iota_Q^G\rho$ to be $(H,\chi,\lambda^G)$-relatively supercuspidal (see \Cref{defn-rsc}\eqref{rsc-defn-1}).
  The main result of \Cref{sec-not-rsc} is the following theorem.
  
  \begin{thm*}[{\Cref{thm-twisted-not-rsc}}]
Let $\chi$ be an unramified quasi-character of $H$.
Let $Q=LU$ be a $\theta$-stable parabolic subgroup of $G$ with $\theta$-stable Levi subgroup $L$ and unipotent radical $U$.
Let $\rho$ be an irreducible representation of $L$ and assume that $\delta_Q^{1/2}\rho$ is $(L^\theta,\delta_{Q^\theta}\chi\vert_{L^\theta})$-distinguished. 
Let $\pi = \iota_Q^G\rho$.
Let $\lambda \in \Hom_{L^\theta}(\delta_Q^{1/2}\rho,\delta_{Q^\theta}\chi\vert_{L^\theta})$ be nonzero and construct $\lambda^G \in \Hom_H(\pi,\chi)$ via \Cref{lem-hom-injects}. 
If $\pi$ is $(H,\chi,\lambda^G)$-relatively supercuspidal, then $\delta_Q^{1/2}\rho$ must be $(L^\theta, \delta_{Q^\theta}\chi\vert_{L^\theta}, \lambda)$-relatively supercuspidal.
\end{thm*}

Of particular interest is the case that $\chi$ is trivial;  we sate this special case of \Cref{thm-twisted-not-rsc} below as \Cref{thm-not-rsc}.

\begin{rmk}[The open orbit]
	It is natural to ask about non-vanishing of the invariant forms $\lambda'_{N,\chi}$  where $P=MN$ is a $\theta$-split parabolic subgroup, $\rho$ is $M^\theta$-distinguished, and $\lambda' \in \Hom_H(\iota_P^G\rho,\chi)$ arises from the open orbit in $P \backslash G / H$ (such linear functionals were constructed and studied in \cite{blanc--delorme2008}).  This question has been addressed by Carmona and Delorme in \cite{carmona--delorme2014} and we refer the reader to their account.
\end{rmk}

\begin{rmk}
Many of the results of \Cref{sec-equivariant-kt08,sec-twisted-rsc-cond} (in particular, \Cref{prop-twisted-rPlambda} and \Cref{thm-twisted-rsc-condition}) were obtained previously by P.~Delorme \cite{delorme2010} via the methods of Lagier \cite{lagier2008}. Our proofs of these results employ the methods of Kato and Takano \cite{kato--takano2008}. Independently, S.~Takeda has obtained similar results also via the methods of Kato and Takano.\footnote{Takeda's results were announced during the summer of 2018.}  
\Cref{thm-twisted-not-rsc} is new and greatly improves upon \cite[Proposition 6.1.3]{smith-phd2017}.
\end{rmk}

\subsection{Local fields}
Let $F$ be a nonarchimedean local field with odd residual characteristic. We allow $F$ to have positive characteristic.
Let $\of$ be the ring of integers of $F$ and
 fix a uniformizer $\varpi$ of $F$. 
 Let $q$ be the cardinality of the residue field $k_F$ of $F$. 
 Let $|\cdot |_F$ denote the normalized absolute value on $F$ such that $|\varpi|_F = q^{-1}$.  We reserve $|\cdot |$ for the usual absolute value on $\C$.

\subsection{Reductive groups and   symmetric spaces}
Let $\G$ be a connected reductive group defined over $F$.  
Let $\theta$ be an $F$-involution of $\G$.  
Let $\Hbf = \G^\theta$ be the subgroup of $\theta$-fixed points in $\G$.
Write $G = \G(F)$ for the group of $F$-points of $\G$; similarly, $H = \Hbf(F)$.  
The quotient $G / H$ is a   symmetric space.

\begin{rmk}
	We will routinely abuse notation and identify an algebraic group  with its group of $F$-points.  
When the distinction is to be made, we will use boldface for the algebraic group and regular typeface for the group of $F$-points.
\end{rmk}

If $X$ is a subset of a group $G$, then let $N_G(X)$ denote the normalizer of $X$ in $G$ and let $C_G(X)$ denote the centralizer of $X$ in $G$. Given $g \in G$, we write $\Int g$ for the inner automorphism of $G$ given by $\Int g(x) = gxg^{-1}$, for all $x\in G$.

For an $F$-torus $\mathbf A \subset G$, let $A^1$ be the subgroup $\mathbf A(\of)$ of $\of$-points of $A = \mathbf A(F)$.  Let $X^*(\mathbf A)$ be the group of $F$-rational characters of $\mathbf A$.
\subsubsection{Tori and root systems relative to involutions}\label{sec-tori-involution}
An element $g \in G$ is $\theta$-split if $\theta(g) = g^{-1}$.  An $F$-torus $S$ contained in $G$ is $(\theta,F)$-split if $S$ is $F$-split and every element of $S$ is $\theta$-split. 
 
Let $S_0$ be a maximal $(\theta,F)$-split torus of $G$.  Let $ A_0$ be a $\theta$-stable maximal $F$-split torus of $G$ that contains $S_0$ \cite[Lemma 4.5(iii)]{helminck--wang1993}. 
Let $\Phi_0 = \Phi(G,A_0)$ be the root system of $G$ with respect to $A_0$.
Let $W_0 = W(G, A_0) = N_{G}(A_0) / C_{G}(A_0)$ be the Weyl group of $G$ with respect to $A_0$.

The torus $A_0$ is $\theta$-stable, so there is an action of $\theta$ on the $F$-rational characters $X^*(A_0)$; moreover, $\Phi_0$ is a $\theta$-stable subset of $X^*(A_0)$.
Recall that a base of $\Phi_0$ determines a choice of positive roots.

\begin{defn}\label{defn-theta-base}
A base $\Delta_0$ of $\Phi_0$ is called a $\theta$-base if for every positive root $\alpha \in \Phi_0^+$ such that $\theta(\alpha) \neq \alpha$ we have that $\theta(\alpha) \in \Phi_0^-$.   
\end{defn}

Let $r: X^*(A_0) \rightarrow X^*(S_0)$ be the surjective map defined by restriction of ($F$-rational) characters.
Let $\Delta_0$ be a $\theta$-base of $\Phi_0$.
Define $\overline \Phi_0 = r(\Phi_0) \setminus \{0\}$ and $\overline \Delta_0 = r(\Delta_0) \setminus \{0\}$.
The set $\overline \Phi_0$ coincides with $\Phi_0(G, S_0)$ and is referred to the as the restricted root system of $G/H$ \cite[Proposition 5.9]{helminck--wang1993}.
The set $\overline \Delta_0$ is a base of the root system $\overline \Phi_0$.  Note that $\overline \Phi_0$ is not necessarily reduced.
Let $\Phi_0^\theta$ and $\Delta_0^\theta$ be the subsets of $\theta$-fixed roots in $\Phi_0$, respectively $\Delta_0$.
We have that $\overline \Phi_0 = r(\Phi_0 \setminus \Phi_0^\theta)$ and $\overline \Delta_0 = r(\Delta_0 \setminus \Delta_0^\theta)$.
Let $\overline \Theta$ be a subset of $\overline \Delta_0$.  Set $[\overline \Theta] = r^{-1}(\overline \Theta) \cup \Delta_0^\theta$.  Subsets of $\Delta_0$ of the form $[\overline \Theta]$ are called $\theta$-split.  Maximal $\theta$-split subsets of $\Delta_0$ are of the form $[\overline\Delta_0 \setminus \{\bar \alpha \}]$, where $\bar \alpha \in \overline \Delta_0$.

\subsubsection{Parabolic subgroups relative to involutions}\label{sec-pblc-involution}
Let $\Pbf$ be an $F$-parabolic subgroup of $\G$.  We refer to $F$-parabolic subgroups of $\G$ simply as parabolic subgroups.
Let $\Nbf$ be the unipotent radical of $\Pbf$.  The reductive quotient $\mathbf M \cong \Pbf / \Nbf $ is called a Levi factor of $\Pbf$.
We denote by $\delta_P$ the modular character of $P = \Pbf(F)$ given by $\delta_P(p) = |\det \Ad_{\style n} (p) |_F$, where $\style n$ is the Lie algebra of $\Nbf $.  We also refer to $M = \Mbf(F)$ as a Levi subgroup of $G$.

Let $M$ be a Levi subgroup of $G$.  Let $A_M$ denote the $F$-split component of the centre of $M$.  The $(\theta,F)$-split component of $M$, denoted by $S_M$, is the largest $(\theta,F)$-split torus of $M$ that is contained in $A_M$.  
More precisely,
\begin{align*}
 \Sbf_\Mbf & = \left( \{ a \in \Abf_\Mbf : \theta(a) = a^{-1} \} \right)^\circ,	
\end{align*}
and $S_M = \Sbf_\Mbf(F)$, where $(\cdot)^\circ$ denotes the Zariski-connected component of the identity.

\begin{defn}\label{defn-theta-split-pblc}
	A parabolic subgroup $P$ of $G$ is $\theta$-split if $\theta(P)$ is opposite to $P$, in which case $M = P \cap \theta(P)$ is a $\theta$-stable Levi subgroup of $P$.
\end{defn}

If $\Theta \subset \Delta_0$ is $\theta$-split, then the $\Delta_0$-standard parabolic subgroup $P_\Theta$ is $\theta$-split.
Let $\Phi_\Theta$ be the subsystem of $\Phi_0$ generated by $\Theta$.
The standard parabolic subgroup $P_\Theta$ has unipotent radical $N_\Theta$ generated by the root subgroups $N_\alpha$, where $\alpha \in \Phi_0^+ \setminus \Phi_\Theta^+$.
The standard Levi subgroup $M_\Theta$ of $P_\Theta$ is equal to the centralizer in $G$ of the $F$-split torus $A_\Theta = \left( \bigcap_{\alpha \in \Theta} \ker \alpha \right)^\circ$.
Any $\Delta_0$-standard $\theta$-split parabolic subgroup of $G$ arises from a $\theta$-split subset of $\Delta_0$ \cite[Lemma 2.5(1)]{kato--takano2008}.

Let $\Theta \subset \Delta_0$ be $\theta$-split.
The $(\theta,F)$-split component of $M_\Theta$ is equal to
\begin{align*}
	S_\Theta &= \left( \bigcap_{\bar \alpha \in r(\Theta)} \ker (\bar\alpha : S_0 \rightarrow F^\times) \right)^\circ
\end{align*}
For any $0 < \epsilon \leq 1$, define 
\begin{align}\label{eq-split-dominant-part}
S_\Theta ^-(\epsilon) = \{ s \in S_\Theta  : |\alpha(s)|_F \leq \epsilon, \ \text{for all} \ \alpha \in \Delta_0 \setminus \Theta \}.
\end{align}
We write $S_\Theta ^-$ for $S_\Theta ^-(1)$ and refer to $S_\Theta ^-$ as the dominant part of $S_\Theta$.

By \cite[Theorem 2.9]{helminck--helminck1998}, the $\theta$-split subset $\Delta_0^\theta$ 
determines the standard minimal $\theta$-split parabolic subgroup $P_0 = P_{\Delta_0^\theta}$.
Let $N_0$ be the unipotent radical of $P_0$.  The standard Levi subgroup $M_0$ of $P_0$ is the centralizer in $G$ of the maximal $(\theta,F)$-split torus $S_0$.

\begin{lem}[{\cite[Lemma 2.5]{kato--takano2008}}]\label{KT08-lem-2.5}
Let $S_0 \subset A_0$, $\Delta_0$, and $P_0 = M_0N_0$ be as above.
\begin{enumerate}
\item Any $\theta$-split parabolic subgroup $P$ of $G$ is conjugate to a $\Delta_0$-standard $\theta$-split parabolic subgroup by an element $g \in (\Hbf \mathbf M_0)(F)$.
\item If the group of $F$-points of the product $(\Hbf \mathbf M_0)(F)$ is equal to $HM_0$, then any $\theta$-split parabolic subgroup of $G$ is $H$-conjugate to a $\Delta_0$-standard $\theta$-split parabolic subgroup.
\end{enumerate}
\end{lem} 

Let $P = MN$ be a $\theta$-split parabolic subgroup. Pick $g\in (\Hbf \mathbf M_0)(F)$ such that $P = gP_\Theta  g^{-1}$ for some $\theta$-split subset $\Theta \subset \Delta_0$.  
Since $g\in (\Hbf \mathbf M_0)(F)$ we have that $g^{-1}\theta(g) \in M_0= \mathbf M_0(F)$, and
we have $S_M = g S_\Theta g^{-1}$.
For a given $\epsilon >0$, one may extend the definition of $S_\Theta ^-$ in \eqref{eq-split-dominant-part}  to the torus $S_M$. 
Set $S_M^-(\epsilon) = g S_\Theta ^-(\epsilon) g^{-1}$ and define $S_M^- = S_M^-(1)$. 
Write $S_M^1$ to denote the group of $\of$-points $\Sbf_\Mbf(\of)$.

\subsection{Distinguished representations and relative matrix coefficients}
A representation $(\pi,V)$ of $G$ is smooth if for every $v\in V$ the stabilizer of $v$ in $G$ is an open subgroup.  
A smooth representation $(\pi,V)$ of $G$ is admissible if, for every compact open subgroup $K$ of $G$, the subspace $V^K$ of $K$-invariant vectors is finite dimensional. 
All of the representations that we consider are smooth and admissible. 
A quasi-character of $G$ is a one-dimensional representation.  
Let $(\pi,V)$ be a smooth representation of $G$. If $\omega$ is a quasi-character of $Z_G$, then $(\pi,V)$ is called an $\omega$-representation if $\pi$ has central character $\omega$.  

Let $P$ be a parabolic subgroup of $G$ with Levi subgroup $M$ and unipotent radical $N$.  
 Given a smooth representation $(\rho, V_\rho)$ of $M$ we may inflate $\rho$ to a representation of $P$, also denoted $\rho$, by declaring that $N$ acts trivially.
 We define the representation $\iota_P^G \rho$ of $G$ to be the (normalized) parabolically induced representation $\ind_P^G (\delta_P^{1/2} \otimes \rho)$, where $G$ acts by right translation of functions.

Let $(\pi,V)$ be a smooth representation of $G$.  Let $(\pi_N, V_N)$ denote the normalized Jacquet module of $\pi$ along $P$.
Precisely,  $V_N$ is the quotient of $V$ by the $P$-stable subspace $V(N) = \spn\{ \pi(n)v - v : n\in N , v\in V\}$, and the action of $P$ on $V_N$ is normalized by $\delta_P^{-1/2}$. 
The unipotent radical of $N$ acts trivially on $(\pi_N,V_N)$ and we will regard $(\pi_N, V_N)$ as a representation of the Levi factor $M \cong P/ N$ of $P$.

Let $\pi$ be a smooth representation of $G$.  We also let $\pi$ denote its restriction to $H$.   Let $\chi$ be a quasi-character of $H$.  
\begin{defn}\label{defn-dist}
The representation $\pi$ is $(H,\chi)$-distinguished if the space $\Hom_H(\pi,\chi)$ is nonzero.  If $\pi$ is $(H,1)$-distinguished, where $1$ is the trivial character of $H$, then we will simply call $\pi$ $H$-distinguished.
\end{defn}

Let $(\pi,V)$ be a smooth $(H,\chi)$-distinguished representation of $G$, and further assume that $\pi$ has central character $\omega$.
Note that $\omega$ restricted to the intersection of the centre $Z_G$ of $G$ with $H$ must agree with $\chi$ on $Z_G\cap H$.
Let $\lambda \in \Hom_H(\pi,\chi)$ be a nonzero linear form on $V$. Let $v$ be a nonzero vector in $V$.
In analogy with the usual matrix coefficients, define a complex-valued function $\varphi_{\lambda,v}$ on $G$ given by $g \mapsto \ip{\lambda}{\pi(g^{-1})v}$. 
Refer to the functions $\varphi_{\lambda,v}$ as $\lambda$-relative matrix coefficients.
When $\lambda$ is understood, we will drop it from the terminology.

Since the representation $(\pi,V)$ is smooth the relative matrix coefficients $\varphi_{\lambda,v}$ lie in $C^\infty(G)$, for every $v\in V$.
Moreover, since $\pi$ has central character $\omega$, the functions $\varphi_{v,\lambda}$ lie in the subspace $C_\omega^\infty(G)$ of $C^\infty(G)$ consisting of smooth (locally constant) functions $f:G \rightarrow \C$ such that $f(zg) = \omega(z)^{-1}f(g)$, for all $z\in Z_G$ and $g\in G$.
For all $g\in G, z \in Z_G$ and $h \in H$ observe that
\begin{align}\label{eq-support}
\varphi_{\lambda,v}(gzh) & = \ip{\lambda}{\pi(h^{-1}z^{-1}g^{-1})}  = \chi(h^{-1})\omega(z^{-1}) \varphi_{\lambda,v}(g).
\end{align}
Define $C^\infty(G,H,\chi)$ to be the subspace of $C^\infty(G)$ such that
\begin{align}\label{eq-H-chi-fcn}
C^\infty(G,H,\chi) = \{f \in C^\infty(G) : f(gh) = \chi(h^{-1}) f(g) \}.	
\end{align}
The map $v \mapsto \varphi_{\lambda,v} \in C^\infty(G,H,\chi)$ intertwines $(\pi,V)$ with the left-regular representation of $G$ on $C^\infty(G,H,\chi)$.\footnote{The representation of $G$ on $C^\infty(G,H,\chi)$ is the left-regular representation of $G$ on $\ind_H^G\chi$.  Note that our conventions for induced representations (for parabolic induction) would have $f(hg) = \chi(h)f(g)$, rather than $f(gh) = \chi(h^{-1})f(g)$.}
Let $C_\omega^\infty(G,H,\chi)$ be the intersection $C^\infty(G,H,\chi) \cap C_\omega^\infty(G)$.
By assumption, $\pi$ is an $\omega$-representation; therefore, $\varphi_{\lambda,v}$ lies in $C_\omega^\infty(G,H,\chi)$.
By \eqref{eq-support}, it makes sense to consider the support of the functions $\varphi_{\lambda,v}$ modulo $Z_GH$.
Indeed, if $\varphi_{\lambda,v}(g)\neq 0$, then $\varphi_{\lambda,v}(gzh)\neq 0$, for all $z\in Z_G$ and $h\in H$.
Finally, define $C_{\omega,0}^\infty(G,H,\chi)$ to be the space of functions
\begin{align}\label{eq-H-chi-fcn-cmpct}
	\{ f \in C_\omega^\infty(G,H,\chi): \supp(f) \ \text{has compact image in} \ G/Z_GH \}.
\end{align}
Observe that we have the chain of containments
\begin{align*}
	C_{\omega,0}^\infty(G,H,\chi) \subset C_\omega^\infty(G,H,\chi) \subset C_\omega^\infty(G) \subset C^\infty(G).
\end{align*}
We'll write $C_c^\infty(G)$ for the space of smooth compactly supported functions on $G$.  This is consistent with the above notation, since we'll write $C_0^\infty(G)$ to denote the space of smooth functions on $G$ the are compactly supported modulo $Z_G$. By definition, $C_c^\infty(G) \subset C_0^\infty(G)$.

\begin{defn}\label{defn-rsc}
The $\omega$-representation $(\pi,V)$ is said to be
\begin{enumerate}
 \item $(H,\chi,\lambda)$-relatively supercuspidal if and only if all of the $\lambda$-relative matrix coefficients are compactly supported modulo $Z_GH$, that is, $\varphi_{\lambda,v} \in C_{\omega,0}(G,H,\chi)$, for every $v\in V$. \label{rsc-defn-1}  
\item $(H,\chi)$-relatively supercuspidal if and only if $\pi$ is $(H,\chi,\lambda)$-relatively supercuspidal for every $\lambda \in \Hom_H(\pi,\chi)$.
\end{enumerate}
When $\chi=1$ is trivial, we drop it from the notation.
\end{defn}

\begin{rmk}[$H$-distinction]
If $\chi =1$, then since $\lambda$ is $H$-invariant the functions $\varphi_{\lambda,v}$ descend to well-defined functions on the quotient $G/H$.
In this case, $C^\infty(G,H,1) \cong C^\infty(G/H)$, $C_\omega^\infty(G,H,1) \cong C_\omega^\infty(G/H)$ and $C_{\omega,0}^\infty(G,H,1) \cong C_{\omega,0}^\infty(G/H)$ is the subspace of functions in $C_\omega^\infty(G/H)$ with compact support modulo $Z_GH$.
\end{rmk}

Let $Q = LU$ be a $\theta$-stable parabolic subgroup with $\theta$-stable Levi factor $L$ and unipotent radical $U$.  Note that the identity component of $Q^\theta = L^\theta U^\theta$ is a parabolic subgroup of $H^\circ$, with the expected Levi decomposition \cite{helminck--wang1993}, \cite[Lemma 3.1]{gurevich--offen2016}.  Let $\mu$ be a positive quasi-invariant measure on the (compact) quotient $Q^\theta \backslash H$ \cite[Theorem 1.21]{bernstein--zelevinsky1976}.

\begin{lem}\label{lem-hom-injects}
Let $\rho$ be a smooth representation of $L$ and let $\pi = \iota_Q^G \rho$. 
The map $\lambda \mapsto \lambda^G$ is an injection of $\Hom_{L^\theta}( \delta_Q^{1/2}\rho, \delta_{Q^\theta}\chi\vert_{L^\theta})$ into $\Hom_H(\pi,\chi)$, where $\lambda^G$ is given explicitly by
\begin{align}\label{def-lambdaG}
\ip{\lambda^G}{f} &= \int_{Q^\theta \backslash H} \ip{\lambda}{\chi(h)^{-1}f(h)} \ d\mu(h)
\end{align}
for any function $f$ in the space $V_\pi$ of $\pi$.  
\end{lem}

 
\begin{proof}
 	Let $I(\rho,QH)$ denote the space of locally constant functions $\phi$ on $QH$ taking values in the space $V_\rho$ of $\rho$ and satisfying:
\begin{align*}
\phi(luh) = \delta_Q^{1/2}(l)\rho(l) \phi(h),
\end{align*}
for  all $l\in L, u\in U$ and $h\in H$.
Restriction of functions in $I(\rho, QH)$ to $H$ gives an equivalence between the $H$-representations $I(\rho, QH)$ and $\cind_{Q^\theta}^H [(\delta_{Q}^{1/2} \rho)\vert_{L^\theta}]=\ind_{Q^\theta}^H [(\delta_{Q}^{1/2} \rho)\vert_{L^\theta}]$, where $\cind$ denotes compact induction.
Applying the exact functor $\Hom_H(\cdot,\chi)$ we obtain 
\begin{align*}
\Hom_H(I(\rho,QH),\chi) \cong \Hom_H(\cind_{Q^\theta}^H [(\delta_{Q}^{1/2} \rho)\vert_{L^\theta}], \chi).
\end{align*}
By \cite[Proposition 2.29]{bernstein--zelevinsky1976}, 
\begin{align*}
\Hom_H(\cind_{Q^\theta}^H [(\delta_{Q}^{1/2} \rho)\vert_{L^\theta}] ,\chi) 
 & \cong 
 \Hom_{L^\theta}(\delta_{Q^\theta}^{-1}\delta_Q^{1/2}\rho, \chi\vert_{L^\theta})\\
  & = \Hom_{L^\theta}(\delta_Q^{1/2}\rho, \delta_{Q^\theta}\chi\vert_{L^\theta});
\end{align*}
moreover, the isomorphism is given explicitly by sending $\lambda$ in the space $\Hom_{L^\theta}(\delta_Q^{1/2}\rho, \delta_{Q^\theta}\chi\vert_{L^\theta})$ to the linear functional on $I(\rho,QH)$ defined by the (convergent) integral
\begin{align*}
\int_{Q^\theta\backslash H} \ip{\lambda}{\chi(h)^{-1} \phi(h)}\ d\mu(h),	
\end{align*}
for any $\phi \in I(\rho,QH)$.
In addition, since $Q$ is $\theta$-stable we have that $QH$ is closed in $G$ \cite[Lemma 1.7]{helminck--wang1993}, and the map $\phi \mapsto \phi\vert_{QH}$ is a surjective $H$-morphism from $V_\pi$ to $I(\rho, QH)$; therefore, we have an injection
\begin{align*}
\Hom_H(I(\rho,QH),\chi) \hookrightarrow \Hom_H(\pi,\chi).
\end{align*}
Taking into account the above isomorphism between $\Hom_H(I(\rho,QH),\chi)$ and $\Hom_{L^\theta}(\delta_Q^{1/2}\rho, \delta_{Q^\theta}\chi\vert_{L^\theta})$, we've obtained an explicit injection
\begin{align*}
\Hom_{L^\theta}(\delta_Q^{1/2}\rho, \delta_{Q^\theta}\chi\vert_{L^\theta}) \hookrightarrow \Hom_H(\iota_Q^G\rho,\chi)	
\end{align*}
 given by $\lambda \mapsto \lambda^{G}$, as claimed.
\end{proof}

\begin{rmk}
The invariant form $\lambda^G$ is the linear functional obtained from $\lambda$ via the closed orbit in $Q \backslash G / H$ and the Mackey theory.  See \cite[Proposition 7.1]{offen2017} for more on this perspective.
\end{rmk}


\begin{cor}\label{cor-hom-injects}
Suppose that $\delta_Q^{1/2}$ restricted to $L^\theta$ is equal to $\delta_{Q^\theta}$. The map $\lambda \mapsto \lambda^G$ is an injection of $\Hom_{L^\theta}(\rho, 1)$ into $\Hom_H(\iota_Q^G\rho,1)$. In particular, if $\rho$ is $L^\theta$-distinguished, then $\iota_Q^G\rho$ is $H$-distinguished.
\end{cor}

\section{Linear funtionals on Jacquet modules}
\subsection{$M^\theta$-invariant linear forms on Jacquet modules}\label{sec-invariant-kt08}
Let $(\pi,V)$ be an admissible $H$-distinguished representation of $G$. Let $\lambda \in \Hom_H(\pi,1)$ be a nonzero $H$-invariant linear form on $V$. 
Let $P$ be a $\theta$-split parabolic subgroup of $G$.   Let $N$ be the  unipotent radical of $P$, and let $M= P \cap \theta(P)$ be a $\theta$-stable Levi factor of $P$.  
Independently, Kato--Takano  and Lagier have defined an $M^\theta$-invariant linear form $\lambda_N$ on the Jacquet module $(\pi_N, V_N)$.

Fix a maximal $\theta$-stable $F$-split torus $A_0$ containing a maximal $(\theta,F)$-split torus $S_0$. Let $\Delta_0$ be a $\theta$-base of the root system $\Phi_0$ of $G$ with respect to $A_0$.  
Assume that $P=P_\Theta$ is a $\Delta_0$-standard $\theta$-split parabolic subgroup of $G$ corresponding to a $\theta$-split subset $\Theta$ of $\Delta_0$.  
In order to define the $M^\theta$-invariant form $\lambda_N$, we first require the following result (\textit{cf.}~\cite[Proposition 1.4.4]{Casselman-book}).

\begin{lem}[{\cite[Lemma 4.3]{kato--takano2008}}]\label{lem-adapted-co-family}  
There exists a decreasing sequence $\{K_n\}_{n\geq 0}$ of $\theta$-stable compact open subgroups of $G$ satisfying the following properties.
\begin{enumerate}
\item The family $\{K_n\}_{n\geq 0}$ gives a neighbourhood base of the identity in $G$.
\item For each $n\geq 1$, the group $K_n$ is normal in $K_0$ and the quotient group $K_n/ K_{n+1}$ is a finite abelian $p$-group, where $p$ is the characteristic of the residue field $k_F$.
\item For each $K_n$, $n\geq1$, and each $\Delta_0$-standard $\theta$-split parabolic subgroup $P_\Theta$ of $G$, $K_n$ has Iwahori factorization with respect to $P_\Theta$,
and 
\begin{align*}
s (N_\Theta \cap K_n) s^{-1} \subset (N_\Theta \cap K_n), \quad s^{-1} (N_\Theta^{\op} \cap K_n) s \subset (N_\Theta^{\op} \cap K_n),
\end{align*}
for all $s \in S_\Theta^-$.
\item For each $\Delta_0$-standard $\theta$-split parabolic subgroup $P_\Theta$ of $G$, the family $\{M_\Theta \cap K_n\}_{n\geq 0}$ satisfies properties (1)--(3) for the group $M_\Theta$.
\end{enumerate}
\end{lem}

We refer to a family of $\theta$-stable compact open subgroups $\{K_n\}_{n\geq 0}$ satisfying (1)--(4) of \Cref{lem-adapted-co-family} as adapted to the data $(S_0,A_0,\Delta_0)$.
The construction of $\lambda_N$ depends on Casselman's Canonical Lifting \cite[Proposition 4.1.4]{Casselman-book} along a family of $\theta$-stable compact open subgroups $\{K_n\}_{n\geq 0}$ adapted to $(S_0,A_0,\Delta_0)$.  

We now recall Casselman's Canonical Lifting following \cite[$\S5.2$]{kato--takano2008}.
For any compact open subgroup $K$ of $G$,  the projection operator $\mathcal P_K$ from $V$ to the subspace of $V^K$ of $K$-fixed vectors is given by the integral
\begin{align*}
\mathcal P_K(v) & = \frac{1}{\mu_K(K)} \int_K \pi(k) v \ d\mu_K,
\end{align*}
for any $v\in V$, where $\mu_K$ is a fixed Haar measure on $K$.
The subspace $V(N)$ of $V$ is equal to the union, over compact open subgroups $N_1$ of $N$, of the sets
\begin{align*}
V_0(N_1) & = \left\{ v \in V : \int_{N_1} \pi(n)v \ d\mu_N = 0 \right\},
\end{align*}
where $\mu_N$ is Haar measure on $N$.
Fix an adapted family $\{K_n\}_{n\geq 0}$ as in \Cref{lem-adapted-co-family}. 
Let $[v] \in V_N$
and take a $\theta$-stable compact open subgroup $K = K_n$ from the adapted family such that $[v] \in (V_N)^{M\cap K}$.
There exists a compact open subgroup $N_0$ of $N$ such that $V^K \cap V(N)$ is contained in $V_0(N_0)$.
By \cite[Lemma 2.8]{kato--takano2008}\footnote{Since $S_M^-(\epsilon) \subset A_M^-(\epsilon)$, one may replace $A_M^-(\epsilon)$ by in $S_M^-(\epsilon)$ Casselman's argument in \cite[Proposition 4.1.4]{Casselman-book}.}, there exists a positive real number $0 < \epsilon \leq 1$ such that $sN_0 s^{-1}$ is contained in $N  \cap K$, for all $s\in S_M^-(\epsilon)$.
For any $s\in S_M^-(\epsilon)$, the projection of $V$ onto $V_N$ induces an isomorphism from $\mathcal P_K \left(\pi(s) V^K\right)$ to $(V_N)^{M\cap K}$.
The vector $v \in \mathcal P_K (\pi(s) V^K)$ that satisfies $v + V(N) = [v]$ is called the {canonical lift} of $[v]$ with respect to $K$.  Casselman proves that $v$ depends on the choice of $K$, but the canonical lift $v$ of $[v]$ does not depend on $N_0$ nor $\epsilon$.
By \cite[Proposition 4.1.8]{Casselman-book}, if $v'$ is any other canonical lift of $[v]$ with respect to a subgroup $K'$ of $K$, then $v' \in V^{(M\cap K)(N^{\op}\cap K)}$ and $v = \mathcal P_K(v') = \mathcal P_{N\cap K}(v')$.
The following is \cite[Proposition 5.3]{kato--takano2008}. 

\begin{prop}[Kato--Takano]\label{prop-rP-well-defn}
Let $\lambda$ be an $H$-invariant linear form on an admissible representation $(\pi,V)$ of $G$ and let $P=MN$ be a $\Delta_0$-standard $\theta$-split parabolic subgroup.
\begin{enumerate}
\item For $K = K_n$, $n \geq 1$, in the adapted family $\{K_n\}_{n\geq 0}$ and $v \in V^{(M\cap K)(N^{\op}\cap K)}$, we have that $\ip{\lambda}{v} = \ip{\lambda}{\mathcal P_{N\cap K}(v)}$.
\item Given $[v]\in V_N$, for any two canonical lifts $v, v' \in V$, we have $\ip{\lambda}{v} = \ip{\lambda}{v'}$.
\end{enumerate}
\end{prop}

\begin{defn}[Kato--Takano, Lagier]\label{defn-rPlambda}
Let $P = MN$ be a $\theta$-split parabolic subgroup of $G$ with $\theta$-stable Levi subgroup $M = P \cap \theta(P)$,  and unipotent radical $N$.
Let $\lambda \in \Hom_H(\pi,1)$ be nonzero. 
Define the linear form $\lambda_N$ on $V_N$ by declaring that 
\begin{align*}
\ip{\lambda_N}{[v]} &= \ip{\lambda}{v},
\end{align*}
where $v \in V$ is any canonical lift of $[v] \in V_N$ with respect to an adapted family $\{K_n\}_{n\geq 0}$.
\end{defn}
By \Cref{prop-rP-well-defn},  the linear form $\lambda_N$ is well defined and does not depend on the choice of $\{K_n\}_{n\geq 0}$, nor on the choice of canonical lift. 

\begin{prop}[Kato--Takano, Lagier]\label{prop-rPlambda}
Let $(\pi,V)$ be an admissible $H$-distinguished representation of $G$.
Let $\lambda \in \Hom_H(\pi,1)$ be nonzero and let $P$ be a $\theta$-split parabolic subgroup of $G$ with unipotent radical $N$ and $\theta$-stable Levi component $M= P \cap \theta(P)$.
\begin{enumerate}
\item The linear functional $\lambda_N: V_N \rightarrow \C$ is $M^\theta$-invariant.
\item The mapping $\Hom_H(\pi,1) \rightarrow \Hom_{M^\theta}(\pi_N, 1)$, sending $\lambda$ to $\lambda_N$, is linear.
\end{enumerate}
\end{prop}

\begin{proof}
	See, for instance, \cite[Proposition 5.6]{kato--takano2008}.
\end{proof}

\begin{thm}[{\cite[Theorem 7.1]{kato--takano2008}}]\label{kt08-rsc-thm}
Let $(\pi,V)$ be an admissible $H$-distinguished representation of $G$ and let $\lambda$ be a nonzero $H$-invariant linear form on $V$. Then, $(\pi,V)$ is $(H,\lambda)$-relatively supercuspidal if and only if $\lambda_N = 0$ for every proper $\theta$-split parabolic subgroup $P$ of $G$.
\end{thm}

\subsection{$(M^\theta, \chi \vert_{M^\theta})$-equivariant linear forms on Jacquet modules}\label{sec-equivariant-kt08}

With only minor modifications to the arguments, the results of \cite{kato--takano2008} can be generalized to $(H,\chi)$-distinguished representations with $\chi$ nontrivial.
We keep the notation of \Cref{sec-invariant-kt08}.

Let $(\pi,V)$ be an admissible $(H,\chi)$-distinguished representation of $G$.
Let $\lambda \in \Hom_H(\pi,\chi)$ be nonzero.
Let $P = MN$ be a $\theta$-split parabolic subgroup of $G$ with $\theta$-stable Levi factor $M = P \cap \theta(P)$ and unipotent radical $N$.
Let $S_M$ be the $(\theta,F)$-split component of $M$.

As above, we may regard $P$ as standard with respect to some choice of data $(S_0, A_0, \Delta_0)$.
Let $\{K_n\}_{n\geq 0}$ be an adapted family of $\theta$-stable compact open subgroups as in \Cref{lem-adapted-co-family}.
Since $\{K_n\}_{n\geq 0}$ forms a neighbourhood basis at the identity of $G$, with each $K_n$ $\theta$-stable, and $\chi:H\rightarrow \C^\times$ is continuous,  we have that for all $n$ large enough $K_n \cap H$ is contained in $\ker \chi$.
The following generalizes \Cref{prop-rP-well-defn} (\cite[Proposition 5.3]{kato--takano2008}) and the proof is the same, as long as one takes $n$ large enough so that $K_n \cap H \subset \ker \chi$ and notices that \cite[Lemma 4.6]{kato--takano2008} actually gives $N\cap K \subset (H\cap K)(M\cap K)(N^{\op} \cap K)$.  

\begin{prop}\label{prop-twisted-rP-well-defn}
	Let $(\pi,V)$ be an admissible $(H,\chi)$-distinguished representation of $G$.  Let $\lambda \in \Hom_H(\pi,\chi)$ be nonzero.  Let $P=MN$ be a $\Delta_0$-standard $\theta$-stable parabolic subgroup of $G$.
	\begin{enumerate}
		\item Let $K = K_n$ be a member of the adapted family $\{K_n\}_{n\geq 0}$ such that $K\cap H \subset \ker \chi$, then for any $v \in V^{(M\cap K)(N^{\op}\cap K)}$ we have $\ip{\lambda}{v} = \ip{\lambda}{\mathcal P_{N\cap K}(v)}$
		\item For any two canonical lifts $v$ and $v'$ of a fixed $[v] \in V_N$, we have $\ip{\lambda}{v} = \ip{\lambda}{v'}$.
	\end{enumerate}
\end{prop}

\Cref{prop-twisted-rP-well-defn} allows us to define $\lambda_{N,\chi} \in \Hom_{M^\theta}(\pi_N, \chi \vert_{M^\theta})$ in exactly the same manner as $\lambda_N$ is defined in \Cref{defn-rPlambda}.

\begin{defn}\label{defn-twisted-rPlambda}
Let $P = MN$ be a $\theta$-split parabolic subgroup of $G$ with $\theta$-stable Levi subgroup $M$,  and unipotent radical $N$.
Let $\lambda \in \Hom_H(\pi,\chi)$ be nonzero. 
Define the linear form $\lambda_{N,\chi}$ on $V_N$ by declaring that 
\begin{align*}
\ip{\lambda_{N,\chi}}{[v]} &= \ip{\lambda}{v},
\end{align*}
where $v \in V$ is any canonical lift of $[v] \in V_N$ with respect to an adapted family $\{K_n\}_{n\geq 0}$.
\end{defn}
By \Cref{prop-twisted-rP-well-defn}, the linear form $\lambda_{N,\chi}$ is well defined and does not depend on the choice of $\{K_n\}_{n\geq 0}$, nor on the choice of canonical lift. 
Moreover, if $\pi$ is $H$-distinguished and $\lambda \in \Hom_{H}(\pi,1)$, then $\lambda_{N,1} = \lambda_N$.

Proposition 5.5 of \cite{kato--takano2008} (\textit{cf.}~\cite[Th\'eor\`{e}me 2]{lagier2008}) holds with $\lambda_N$ replaced by $\lambda_{N,\chi}$.\footnote{Below, when applying the analogue of  \cite[Proposition 5.5]{kato--takano2008} in the context of $\lambda_{N,\chi}$ we will refer to Kato and Takano's result without further comment.}  Again, to adjust the proof, go far enough into the adapted family $\{K_n\}$ to ensure that $K_n \cap H \subset \ker \chi$.
The next result is the analogue of \Cref{prop-rPlambda} (\cite[Proposition 5.6]{kato--takano2008}) and has also been obtained by Delorme \cite[Section 3.1]{delorme2010} 
 via the methods of Lagier \cite{lagier2008}.

\begin{prop}\label{prop-twisted-rPlambda}
	Let $(\pi,V)$ be an admissible $(H, \chi)$-distinguished representation of $G$.  Let $\lambda \in \Hom_H(\pi,\chi)$ be nonzero.  Let $P = MN$ be a $\theta$-split parabolic subgroup of $G$ with $\theta$-stable Levi subgroup $M = P \cap \theta(P)$ and unipotent radical $N$.
	\begin{enumerate}
		\item The linear functional $\lambda_{N,\chi}: V_N \rightarrow \C$ lies in $\Hom_{M^\theta}(\pi_N, \chi \vert_{M^\theta})$.
		\item The mapping $\Hom_H(\pi,\chi) \rightarrow \Hom_{M^\theta}(\pi_N, \chi \vert_{M^\theta})$, sending $\lambda$ to $\lambda_{N,\chi}$ is linear.
	\end{enumerate}
\end{prop}

\begin{proof}
The proof is the same as that of \cite[Proposition 5.6]{kato--takano2008} after taking into account the character $\chi$ and replacing Kato and Takano's $\bar \lambda = \lambda_N \circ \pi_N(m)$ with $\lambda ' = \chi(m)^{-1} \lambda_{N,\chi}\circ \pi_N(m)$.
\end{proof}

The analogue of \cite[Proposition 5.9]{kato--takano2008} (\textit{cf}.~\cite[Proposition 3.9]{delorme2010}) also holds.
\begin{prop}\label{prop-twisted-transitive}
Let $(\pi,V)$ be an admissible $(H, \chi)$-distinguished representation of $G$.  Let $\lambda \in \Hom_H(\pi,\chi)$ be nonzero.
	Let $P' = M'N' \subset P = MN$ be two $\theta$-split parabolic subgroups of $G$ with the indicated Levi decompositions.
	The following diagram commutes:
\[
\begin{tikzcd}
  \Hom_H(\pi,\chi) \arrow[r, "\lambda \mapsto \lambda_{N,\chi}"] \arrow[d, "\lambda \mapsto \lambda_{N',\chi}"]
    & \Hom_{M^\theta}(\pi_N, \chi\vert_{M^\theta}) \arrow[d, "\lambda_{N,\chi} \mapsto (\lambda_{N,\chi})_{N', \chi\vert_{M^\theta}}"] \\
  \Hom_{(M')^\theta}\left(\pi_{N'}, \chi\vert_{(M')^\theta}\right) \arrow[r, equal, "\sim"]
&\Hom_{(M')^\theta}\left((\pi_N)_{N'\cap M}, \chi\vert_{(M')^\theta}\right) \end{tikzcd}
\]
where the bottom isomorphism follows from the transitivity of the Jacquet restriction functor, i.e., $\pi_{N'} \cong (\pi_N)_{N' \cap M}$.
Identifying $\pi_{N'}$ and  $(\pi_N)_{N' \cap M}$, we have that $\lambda_{N',\chi} = (\lambda_{N,\chi})_{N', \chi\vert_{M^\theta}}$.
\end{prop}

\begin{note}
In the bottom right of the diagram in \Cref{prop-twisted-transitive}, we've used that $\left(\chi\vert_{M^\theta}\right)\vert_{(M')^\theta}=\chi\vert_{(M')^\theta}$, i.e, we can restrict the character $\chi$ in stages.	
\end{note}

\begin{rmk}
It appears that, in the proof of \cite[Proposition 5.9]{kato--takano2008} when applying \cite[proposition 5.5(1)]{kato--takano2008}, Kato and Takano use that $S_{M'}^-(\epsilon)$ is contained in $S_M^-(\epsilon)$, which is not true if $M'$ is a proper Levi of $M$.
However, Lagier's result \cite[Th\'eor\`{e}me 2]{lagier2008} may be used, in place of \cite[proposition 5.5(1)]{kato--takano2008}, to prove \cite[Proposition 5.9]{kato--takano2008} and similarly \Cref{prop-twisted-transitive}.
Importantly, \cite[Proposition 5.5(2)]{kato--takano2008} remains true when using \cite[Th\'eor\`{e}me 2]{lagier2008} in place of \cite[Proposition 5.5(1)]{kato--takano2008}.
\end{rmk}

\subsection{Characterizing $(H,\chi,\lambda)$-relative supercuspidality}\label{sec-twisted-rsc-cond}
Here, we use apply the results of \Cref{sec-equivariant-kt08} to generalize the results of \cite[Section 6]{kato--takano2008} to $(H,\chi)$-distinguished representations.
The arguments are exactly the same as in Kato and Takano's work except that we apply the results of \Cref{sec-equivariant-kt08} in place of those in \cite[Section 5]{kato--takano2008} (\textit{cf}.~\Cref{sec-invariant-kt08} \textit{loc cit.}); therefore, we'll only note when the arguments need to be adjusted.
The main result is the following generalization of \cite[Theorem 6.2]{kato--takano2008}.

\begin{thm}\label{thm-twisted-rsc-condition}
	Let $(\pi,V)$ be an admissible $(H,\chi)$-distinguished representation of $G$ and let $\lambda \in \Hom_H(\pi,\chi)$ be nonzero.
	Then $(\pi,V)$ is $(H,\chi,\lambda)$-relatively supercuspidal if and only if $\lambda_{N,\chi} = 0$, for every proper $\theta$-split parabolic subgroup $P=MN$ of $G$.
\end{thm}

\begin{rmk}
Of course, if $\chi =1$ and $\pi$ is $H$-distinguished, then \Cref{thm-twisted-rsc-condition} recovers \cite[Theorem 6.2]{kato--takano2008}.
\end{rmk}

Before giving a proof of \Cref{thm-twisted-rsc-condition}, we assemble preliminary results analogous to the results of \cite[Section 6.5]{kato--takano2008}.
Recall that $P_0=M_0N_0$ is the minimal $\Delta_0$-standard $\theta$-split parabolic subgroup of $G$ associated to the data $(S_0,A_0,\Delta_0)$.
If $x\in (\Hbf \Mbf_0)(F)$, define $m_x = x^{-1}\theta(x) \in M_0 = \Mbf_0(F)$.
The twisted involution $\theta_x$ is given by
\begin{align*}
\theta_x(g) &= x^{-1}\theta(xgx^{-1})x = m_x \theta(g) m_x^{-1},
\end{align*}
for all $g\in G$.
The $\theta_x$-fixed points of $G$ are related to the $\theta$-fixed points by the following equality
\begin{align*}
G^{\theta_x} & = x^{-1} G^\theta x = x^{-1} H x.	
\end{align*}

\begin{note}
The maximal $(\theta,F)$-split torus $S_0$ is centralized by $M_0$; therefore, since $x \in (\Hbf\Mbf_0)(F)$, $S_0$ is also a maximal $(\theta_x,F)$-split torus of $G$.
Moreover, any $\Delta_0$-standard $\theta$-split parabolic subgroup of $G$ is $\theta_x$-split.	
The $(\theta,F)$-split and $(\theta_x,F)$-split components for such a parabolic subgroup coincide.
\end{note}

Assume that $P=MN$ is a $\Delta_0$-standard $\theta$-split parabolic subgroup of $G$.
Let $\chi$ be a quasi-character of $H = G^\theta$.
Denote by $\chi'$ the corresponding quasi-character ${}^{x^{-1}}\chi = \chi \circ \Int x^{-1}$ of $G^{\theta_x}$.
Let $\pi$ be an admissible $(G^{\theta_x}, \chi')$-distinguished representation of $G$.
We have a mapping
\begin{align*}
	\Hom_{G^{\theta_x}}(\pi, \chi') & \rightarrow \Hom_{M^{\theta_x}}(\pi_N, \chi' \vert_{M^{\theta_x}}),\\
	\lambda' & \mapsto \lambda'_{N,\chi'}.
\end{align*}

For $x \in (\Hbf\Mbf_0)(F)$, as above, the parabolic subgroup $Q = xPx^{-1}$ is $\theta$-split but possibly nonstandard.
Let $U$ be the unipotent radical of $Q$ and let $L$ be the $\theta$-stable Levi factor $Q\cap \theta(Q)$. 
The $(\theta,F)$-split component of $L$ is $S_L = x S_M x^{-1}$.
As above, for any $(H,\chi)$-distinguished admissible representation $\pi$ of $G$ we have a mapping
\begin{align*}
	\Hom_{H}(\pi,\chi) & \rightarrow \Hom_{L^\theta}(\pi_U, \chi\vert_{L^\theta}). \\
	\lambda & \mapsto \lambda_{U,\chi}
\end{align*}

Observe that $L = xMx^{-1}$ and $L^\theta = x M^{\theta_x} x^{-1}$.
Moreover, the map sending $\lambda \in \Hom_{H}(\pi,\chi)$ to $\lambda ' = \lambda \circ \pi(x) \in \Hom_{G^{\theta_x}}(\pi, \chi')$ is a linear isomorphism.
The following is the direct analogue of \cite[Lemma 6.6]{kato--takano2008}.

\begin{lem}\label{lem-twisted-kt-6-6}
Let $x \in (\Hbf \Mbf_0)(F)$.
Let $P=MN$, and $Q=LU = xPx^{-1}$ be $\theta$-split parabolic subgroups of $G$.	 
Let $\chi$ be a quasi-character of $H$ and let $\chi'=\chi\circ \Int x^{-1}$.
Let $(\pi,V)$ be an $(H,\chi)$-distinguished representation of $G$.
Given $\lambda \in \Hom_{H}(\pi,\chi)$, set $\lambda ' = \lambda \circ \pi(x) \in \Hom_{G^{\theta_x}}(\pi, \chi')$.
The relation 
\begin{align}\label{eq-twisted-kt-6-6}
	\ip{\lambda'_{N,\chi'}	}{[v]_N} & = \ip{\lambda_{U,\chi}}{[\pi(x)v]_U}
\end{align}
holds for every $\lambda \in \Hom_H(\pi,\chi)$ and $v\in V$.\footnote{We use $[v]_N$ to denote the image of $v$ in $V_N$ and $[v]_U$ to denote the image of $v$ in $V_U$.}
\end{lem}
\begin{proof}
The map $\pi(x): V \rightarrow V$ maps $V(N)$ isomorphically onto $V(U)$; therefore, $\pi(x)$ induces a linear isomorphism $\overline{\pi(x)}: V_N \rightarrow V_U$ of Jacquet modules.
Moreover, $\overline{\pi(x)}[v]_N = [\pi(x) v]_U$, for every $v\in V$.
To prove the equality in \eqref{eq-twisted-kt-6-6}, we will apply \cite[Proposition 5.5(2)]{kato--takano2008}.
Set $\bar \lambda = \lambda_{U,\chi} \circ \overline{\pi(x)}$.
For $s\in S_M$, we have that
\begin{align*}
\delta_P^{1/2}(s)\ip{\bar\lambda}{\pi_N(s)[v]_N} 
& = \ip{\lambda_{U,\chi}}{\overline{\pi(x)} (\delta_P^{1/2}(s) \pi_N(s)[v]_N)}\\
& = \ip{\lambda_{U,\chi}}{\overline{\pi(x)}[\pi(s)v]_N}\\
& = \ip{\lambda_{U,\chi}}{[\pi(x)\pi(s)v]_U} \\
& = \ip{\lambda_{U,\chi}}{\delta_Q^{1/2}(xsx^{-1})\pi_U(xsx^{-1})[\pi(x)v]_U}.
\end{align*}
By \cite[Proposition 5.5(1)]{kato--takano2008}, applied with respect to $Q=LU$, there exists $0 < \epsilon \leq 1$ with 
\begin{align*}
\ip{\lambda_{U,\chi}}{\delta_Q^{1/2}(xsx^{-1})\pi_U(xsx^{-1})[\pi(x)v]_U}
& = \ip{\lambda}{\pi(xsx^{-1})\pi(x)v} \\
& = \ip{\lambda}{\pi(x)\pi(s)v} \\
& = \ip{\lambda'}{\pi(s)v},	
\end{align*}
for every $s \in S_M^-(\epsilon)$.\footnote{Recall that $S_L = x S_M x^{-1}$.}
By \cite[Proposition 5.5(2)]{kato--takano2008}, $\bar \lambda = \lambda_{U,\chi} \circ \overline {\pi(x)}$ is equal to $\lambda'_{N,\chi'}$, 
which completes the proof.
\end{proof}

\begin{rmk}\
\begin{enumerate}
	\item In the second and fourth lines of the first calculation in the proof of \Cref{lem-twisted-kt-6-6}, we used that $\pi_N(s)[v]_N = \delta_P^{-1/2}(s)[\pi(s)v]_N$, i.e., we work with normalized Jacquet modules.
	\item Explicitly, $\lambda'_{N,\chi'}$ is equal to $(\lambda \circ \pi(x))_{N,{}^{x^{-1}}\chi}$.
\end{enumerate}
\end{rmk}

The next result is the analogue of \cite[Lemma 6.7]{kato--takano2008}.

\begin{lem}\label{lem-twisted-kt-6-7}
 Let $\lambda \in \Hom_H(\pi,\chi)$ be a nonzero $(H,\chi)$-equivariant linear functional on an admissible $(H,\chi)$-distinguished representation $(\pi,V)$ of $G$. 
Let $P=MN$ be the $\theta$-split parabolic subgroup $xP_\Theta x^{-1}$ of $G$, where $\Theta$ is a $\theta$-split  subset of $\Delta_0$, $P_\Theta=M_\Theta N_\Theta$, and $x \in  (\Hbf \Mbf_0)(F)$.  
Let $\style C$ be a compact subset of $G$.
 If $\lambda_{N,\chi}= 0$, then for all $v\in V$ there exists a positive real number $\epsilon \leq 1$, depending on $\Theta$ and $x$, such that the relative matrix coefficient $\varphi_{\lambda, v}$ vanishes identically on $\style C S_\Theta^+(\epsilon) x^{-1} H$.
\end{lem}

\begin{proof}
	Let $v\in V$ be nonzero.
	Let $k\in \style C$, $s\in S_\Theta^-(\epsilon)$, and $h\in H$.
	By definition,  $s^{-1} \in S_\Theta^+(\epsilon)$.
	Observe that 
	\begin{align*}
		\varphi_{\lambda,v}(ks^{-1}x^{-1}h) 	
		& = \ip{\lambda}{\pi(h^{-1}xsk^{-1}v)} \\
		& = \chi(h^{-1}) \ip{\lambda}{\pi(x)\pi(s)\pi(k^{-1}) v}\\
		& = \chi(h^{-1})\ip{\lambda'}{\pi(s)\pi(k^{-1})v},
	\end{align*}
where $\lambda' = \lambda \circ \pi(x) \in \Hom_{G^{\theta_x}}(\pi,\chi')$, and $\chi' = {}^{x^{-1}}\chi$.
Recall that, since $\style C$ is compact, $\pi(k^{-1})v$ remains in a finite dimensional subspace of $V$ for every $k \in \style C$.
By \cite[Proposition 5.5(1)]{kato--takano2008} for $\lambda'$ and $P_\Theta=M_\Theta N_\Theta$, there exists $0 < \epsilon \leq 1$ such that  
	\begin{align*}
		\chi(h^{-1})\ip{\lambda'}{\pi(s)\pi(k^{-1})v}
		& = \chi(h^{-1}) \delta_{P_\Theta}^{1/2}(s) \ip{\lambda'_{N_\Theta, \chi'}}{\pi_{N_\Theta}(s)[\pi(k^{-1})v]_{N_\Theta}} \\
		& = \chi(h^{-1})\ip{\lambda'_{N_\Theta, \chi'}}{[\pi(s)\pi(k^{-1})v]_{N_\Theta}},
\end{align*}
for every $s \in S_\Theta^-(\epsilon)$ and $k \in \style C$.
By \Cref{lem-twisted-kt-6-6}, we have
\begin{align*}
\chi(h^{-1})\ip{\lambda'_{N_\Theta,\chi'}}{[\pi(s)\pi(k^{-1})v]_{N_\Theta}} 
		& = \chi(h^{-1})\ip{\lambda_{N,\chi}}{[\pi(x)\pi(s)\pi(k^{-1})]_N}\\
		& = 0,
\end{align*}
where the last equality holds by assumption.
In particular, there is $0< \epsilon \leq 1$ such that $\varphi_{\lambda,v}(ks^{-1}x^{-1}h)=0$ for every $k\in \style C$, $s\in S_\Theta^-(\epsilon)$, and $h \in H$.
\end{proof}

\begin{proof}[Proof of \Cref{thm-twisted-rsc-condition}]
The proof of the ``only if" part follows exactly as that of \cite[Theorem 6.2]{kato--takano2008} (see \cite[Section 6.3]{kato--takano2008}).	Moreover, with \Cref{lem-twisted-kt-6-6,lem-twisted-kt-6-7} in place of \cite[Lemma 6.6]{kato--takano2008} and \cite[Lemma 6.7]{kato--takano2008}, the proof of the ``if" part of \Cref{thm-twisted-rsc-condition} goes through exactly as that of \cite[Theorem 6.2]{kato--takano2008} (see \cite[Section 6.8]{kato--takano2008}).
\end{proof}

\subsection{An $(H,\chi)$-relative subrepresentation theorem}\label{sec-twisted-subrep}
We now give a generalization of the relative Jacquet Subrepresentation Theorem \cite[Theorem 7.1]{kato--takano2008}.  To do so, we first need the $(H,\chi)$-analogue of \cite[Proposition 1.11]{kato--takano2008}.  With \Cref{thm-twisted-rsc-condition} in place, the proofs of both results follow by the same arguments as in the work of Kato and Takano. For completeness, we include the proofs here.

\begin{prop}\label{prop-twisted-kt08-1-11}
Let $(\pi,V)$ be a finitely generated $(H,\chi)$-distinguished representation of $G$ and let $\lambda \in \Hom_H(\pi,\chi)$ be nonzero.
If $(\pi,V)$ is $(H,\chi,\lambda)$-relatively supercuspidal, then $(\pi,V)$ has a nontrivial $(H,\chi)$-distinguished irreducible quotient.	
\end{prop}

\begin{proof}
	Let $T: V \rightarrow C^\infty(G,H,\chi)$ be the intertwining operator that sends $v\in V$ to the $\lambda$-relative matrix coefficient $\varphi_{\lambda,v}$.
	By assumption, $T$ has image contained in $C_0^\infty(G,H,\chi)$.  The quotient $V/ \ker T \cong \ran T \subset C_0^\infty(G,H,\chi)$ is finitely generated and thus finite length \cite[Theorem 6.3.10]{Casselman-book}.
	By \cite[Lemma 1.7]{kato--takano2008}, there exists a quotient $(\rho,W)$ of $V / \ker T$ which is an $\omega$-representation for some quasi-character $\omega$ of $G$; moreover, $(\rho,W)$ is equivalent to a subrepresentation of $C_0^\infty(G,H,\chi)$.
	Note that $(\rho,W)$ is also a quotient of $V$.
	By \cite[Lemma 1.10(2)]{kato--takano2008}, there exists a positive valued quasi-character $\xi: G \rightarrow \R_{0>}$ such that $\xi\vert_H = 1$ and $\xi\vert_{Z_G} = |\omega(\cdot)|^{-1}$.
	The representation $\xi\otimes \rho$ is $(H,\chi)$-distinguished.
	We may regard $\xi\otimes\rho$ as a subrepresentation of $C_{0,\omega_u}^\infty(G,H,\chi)$, where $\omega_u = \omega \otimes |\omega(\cdot)|^{-1}$ is a unitary character.
	Thus $(\xi\otimes\rho,W)$ is unitarizable and decomposes into a direct sum of finitely many irreducible subrepresentations \cite[Proposition 2.1.14]{Casselman-book}.  The decomposition for the action of $\xi\otimes \rho$ on $W$ also holds for the action of $\rho$ on $W$.  At least one of the irreducible direct summands of $(\rho,W)$ must be $(H,\chi)$-distinguished and this is the desired quotient of $(\pi,V)$.
	\end{proof}

\begin{thm}\label{thm-twisted-kt08-7-1}
Let $(\pi,V)$ be an irreducible admissible $(H,\chi)$-distinguished representation of $G$.  There exists a $\theta$-split parabolic subgroup $P=MN$ of $G$ and an irreducible $(M^\theta, \chi\vert_{M^\theta})$-relatively supercuspidal representation $(\rho,W)$ of $M$ such that $\pi$ is equivalent to a subrepresentation of $\iota_P^G\rho$.
\end{thm}

\begin{proof}
	Argue by induction on the rank $k$ of the maximal $(\theta,F)$-split tori of $G/Z_G$. If $k=0$, then by \cite[Proposition 4.7]{helminck--wang1993} $G/Z_GH$ is compact.  In this case, every irreducible $(H,\chi)$-distinguished representation of $G$ is $(H,\chi)$-relatively supercuspidal.  Thus, we assume that $k>0$.  If $\pi$ is $(H,\chi)$-relatively supercuspidal, then there is nothing to do.  Otherwise, there exists a nonzero element $\lambda \in \Hom_H(\pi,\chi)$ such that $\pi$ is not $(H,\chi,\lambda)$-relatively supercuspidal.  By \Cref{thm-twisted-rsc-condition}, there exists a proper $\theta$-split parabolic subgroup $P'=M'N'$ of $G$ such that $\lambda_{N',\chi}\neq 0$.  Let $Q = LU$ be minimal among proper $\theta$-split parabolic subgroups such that $\lambda_{U,\chi} \neq 0$. By \Cref{prop-twisted-transitive} and \Cref{thm-twisted-rsc-condition}, the admissible representation $(\pi_U,V_U)$ of $L$ is $(L^\theta, \chi\vert_{L^\theta}, \lambda_{U,\chi})$-relatively supercuspidal.  By \Cref{prop-twisted-kt08-1-11}, there exists an irreducible $(L^\theta, \chi\vert_{L^\theta})$-distinguished quotient $\rho'$ of $\pi_U$.  By Frobenius Reciprocity, there is a natural bijection
	\begin{align*}
	\Hom_G(\pi,\iota_Q^G \rho') & \simeq \Hom_L(\pi_U,\rho') \neq 0,	
	\end{align*}
	and $\pi$ is equivalent to an irreducible subrepresentation of $\iota_Q^G\rho'$.
	Note that, since $Q$ is proper in $G$, the rank of the maximal $(\theta,F)$-split tori in $L/Z_L$ is strictly less than $k$.
	By induction, applied to $L$ and $\rho'$, there exists a $\theta$-split parabolic subgroup $P=MN$ of $G$ contained in $Q$ and an $(M^\theta, \chi\vert_{M^\theta})$-relatively supercuspidal representation $\rho$ of $M=M\cap L$ such that $\rho'$ is equivalent to a subrepresentation of $\iota_{L\cap P}^L \rho$.  By the transitivity of parabolic induction, $\pi$ is equivalent to a subrepresentation of $\iota_Q^G (\iota_{L\cap P}^L\rho) \cong \iota_P^G\rho$.
\end{proof}
\section{Representations that are not relatively supercuspidal}\label{sec-not-rsc}
Let $\chi$ be a quasi-character of $H = G^\theta$.
Our goal is to study the support of relative matrix coefficients defined with respect to the $(H,\chi)$-equivariant linear forms $\lambda^G$ produced via \Cref{lem-hom-injects}.  Recall that a smooth representation $(\pi,V)$ of $G$ is $(H,\chi,\lambda^G)$-relatively supercuspidal if and only if all of the $\lambda^G$-relative matrix coefficients are compactly supported modulo $Z_GH$ (\Cref{defn-rsc}). 

\begin{note}
In this section, we must restrict our attention to unramified quasi-characters of $H$.  The author hopes to relax this restriction on $\chi$ in future work.
\end{note}
\subsection{Support of functions on $G/Z_GH$}\label{sec-support}
Assume that $\chi$ is an unramified quasi-character of $H$; in particular, $\chi$ is trivial on all maximal compact open subgroups of $H$.
Let $Q=LU$ be a $\theta$-stable parabolic subgroup of $G$ with $\theta$-stable Levi subgroup $L$ and unipotent radical $U$.  Note that the unipotent subgroups $U$ and $U^{\op}$ are $\theta$-stable.  
Moreover, the identity component of $Q^\theta$ is a parabolic subgroup of the identity component $H^\circ$ of $H$ with the expected Levi factorization 
\cite{helminck--wang1993}.  
Let $\rho$ be an irreducible representation of $L$ and assume that $\delta_Q^{1/2}\rho$ is $(L^\theta,\delta_{Q^\theta}\chi\vert_{L^\theta})$-distinguished.
Let $\lambda \in \Hom_{L^\theta}(\delta_Q^{1/2}\rho,\delta_{Q^\theta}\chi\vert_{L^\theta})$ be nonzero.  
Define $\pi = \iota _Q^G \rho$ and construct $\lambda^G \in \Hom_H(\pi,\chi)$ from $\lambda$ via \Cref{lem-hom-injects}.
Our goal is to prove the following.

\begin{thm}\label{thm-twisted-not-rsc}
Let $\chi$ be an unramified quasi-character of $H$.
If the induced representation $\pi = \iota_Q^G\rho$ is $(H,\chi,\lambda^G)$-relatively supercuspidal, then $\delta_Q^{1/2}\rho$ must be $(L^\theta, \delta_{Q^\theta}\chi\vert_{L^\theta}, \lambda)$-relatively supercuspidal.
\end{thm}

\begin{rmk}
	We can rephrase the assumptions on $\rho$ by asking that $\rho$ is  $(L^\theta,\delta_{Q^\theta}(\delta_Q^{-1/2}\chi)\vert_{L^\theta})$-distinguished. In many, but not all, situations the character $\delta_{Q^\theta}\delta_Q^{-1/2}\vert_{L^\theta}$ is trivial, see \cite{smith2018b,smith2018}.  For instance, $\delta_{Q^\theta}\delta_Q^{-1/2}\vert_{L^\theta}$ is not always trivial in the case that $G = \GL_{2n}(F)$ and $H = \Sp_{2n}(F)$ \cite{offen2006b}.
\end{rmk}

If $\chi=1$ is trivial, we identify $\Hom_{L^\theta}(\delta_Q^{1/2}\rho,\delta_{Q^\theta})$ and  $\Hom_{L^\theta}(\rho, \delta_{Q^\theta}\delta_Q^{-1/2}\vert_{L^\theta})$.
As a special case of \Cref{thm-twisted-not-rsc}, we obtain the following.

\begin{thm}\label{thm-not-rsc}
If the induced representation $\pi=\iota_Q^G\rho$ is $(H,\lambda^G)$-relatively supercuspidal, then $\rho$ is $(L^\theta, \chi', \lambda)$-relatively supercuspidal, where  $\chi'$ is the character $\delta_{Q^\theta}\delta_Q^{-1/2}\vert_{L^\theta}$.
\end{thm}

By applying \Cref{thm-twisted-rsc-condition} (the $(H,\chi)$-analogue of \cite[Theorem 6.2]{kato--takano2008}), we can rephrase \Cref{thm-twisted-not-rsc} as the following corollary.

\begin{cor}\label{cor-twisted-not-rsc}
If $(\lambda^G)_{N,\chi} = 0$ for all proper $\theta$-split parabolic subgroups $P=MN$ of $G$, then $\lambda_{N',\delta_{Q^\theta}\chi\vert_{L^\theta}}=0$ for all proper $\theta$-split parabolic subgroups $P' = M'N'$ of $L$.
\end{cor}

Again as a special case, by applying \Cref{kt08-rsc-thm} (\cite[Theorem 6.2]{kato--takano2008}), we can rephrase \Cref{thm-not-rsc} as the following corollary.

\begin{cor}\label{cor-not-rsc}
If  $(\lambda^G)_N = 0$ for all proper $\theta$-split parabolic subgroups $P=MN$ of $G$, then $\lambda_{N',\chi'}=0$ for all proper $\theta$-split parabolic subgroups $P' = M'N'$ of $L$, where $\chi' = \delta_{Q^\theta}\delta_Q^{-1/2}\vert_{L^\theta}$.
\end{cor}

The representation $\pi$ admits a central character $\omega$ since $\pi$ is induced from an irreducible representation.
Since $\pi$ is $(H,\chi)$-distinguished, the character $\omega$ agrees with $\chi$ on $Z_G \cap H$, i.e., $\omega \vert_{Z_G \cap H} = \chi \vert_{Z_G \cap H}$.  
Moreover, since $Z_G$ is the almost direct product $S_G (Z_G \cap H)^\circ$ \cite{helminck1991a}, we have that a function $\phi \in C_{\omega}^\infty(G,H,\chi)$ is compactly supported modulo $Z_GH$ if and only if $\phi$ has compact support modulo $S_GH$.\footnote{Refer to \eqref{eq-H-chi-fcn} and \eqref{eq-H-chi-fcn-cmpct} for a description of $C_{\omega}^\infty(G,H,\chi)$.}
We will now study the support of such a function by considering the behaviour of the function as we approach infinity along non-central $(\theta,F)$-split tori.

Fix a maximal $(\theta,F)$-split torus $S_0$ of $G$ and a $\theta$-base $\Delta_0$ of the root system $\Phi_0 = \Phi(G,A_0)$ of $G$, where $A_0$ is a $\theta$-stable maximal $F$-split torus containing $S_0$.  Define $S_0^+$ to be the set  
\begin{align}
S_0^+ &= \{ s \in S_0 : s^{-1} \in S_0^- \},
\end{align}  
where we recall that $S_0^- = \{ s \in S_0 : |\alpha(s)|_F \leq 1,\ \text{for all}\ \alpha \in \Delta_0\}$.
Given a subset $\Theta \subset \Delta_0$, the sets $S_\Theta^+(\epsilon)$, $0 < \epsilon \leq 1$, are defined analogously (see \eqref{eq-split-dominant-part}),
\begin{align*}
S_\Theta^+(\epsilon) &= \{ s \in S_\Theta : s^{-1} \in S_\Theta^-(\epsilon) \}.
\end{align*}  

\begin{defn}\label{S-infinity}
Let $\Theta \subset \Delta_0$ be a $\theta$-split subset and let $S_\Theta$ be the associated standard $(\theta,F)$-split torus.
The sequence $\{s_j\}_{j\in \N}$ approaches {$S_\Theta$-infinity} if the sequences $\{ |\alpha (s_j)|_F \}_{j\in \N}$ diverge to infinity for all $\alpha \in \Delta_0 \setminus \Theta$.  
\end{defn}
For example, if we take $s \in S_0^+ \setminus S_0^1$ such that $|\alpha(s)|_F>1$ for all $\alpha \in \Delta_0 \setminus \Delta_0^\theta$, then the sequence $\{s^n\}_{n\in \N}$ approaches $S_0$-infinity.  For a non-standard $(\theta,F)$-split torus $S$, we can extend \Cref{S-infinity} to $S$ using that $S = x S_\Theta x^{-1}$, for some $x \in  (\Hbf\mathbf M_0)(F)$ and some standard $(\theta,F)$-split torus $S_\Theta$ (see \Cref{KT08-lem-2.5}).
Precisely, the sequence $\{x s_j x^{-1}\}_{j\in \N} \subset S$ approaches $S$-infinity if and only if the sequence $\{s_j\}_{j\in\N} \subset S_\Theta$ approaches $S_\Theta$-infinity.

To prove \Cref{thm-not-rsc}, we apply the contrapositive of \Cref{lem-twisted-kt-6-7} (the analogue of \cite[Lemma 6.7]{kato--takano2008}), which forms part of the characterization of $(H,\chi)$-relatively supercuspidal representations.

\begin{lem}[Contrapositive of \Cref{lem-twisted-kt-6-7}]\label{lem-twisted-kt-6-7-contrapositive}
	Let $\chi$ be an arbitrary quasi-character of $H$.
	Let $\lambda \in \Hom_H(\pi,\chi)$ be a nonzero $(H,\chi)$-equivariant linear function on an admissible $(H,\chi)$-distinguished representation $(\pi,V)$ of $G$.  Let $P = xP_\Theta x^{-1}$ be a $\theta$-split parabolic subgroup of $G$, where $\Theta \subset \Delta_0$ is $\theta$-split and $x\in (\Hbf \mathbf{M}_0)(F)$. Let $N=xN_\Theta x^{-1}$ be the unipotent radical of $P$ and $M = xM_\Theta x^{-1}$ the standard Levi factor.  Let $\style C$ be a compact subset of $G$.  If there exists $v\in V$ such that, for all $0< \epsilon \leq 1$, the function $\varphi_{\lambda,v}$ is non-zero on some element of $\style C S_\Theta^+(\epsilon) x^{-1} H$, then $\lambda_{N,\chi} \neq 0$.
\end{lem}

Recall that there is an analogue of the Cartan decomposition for   symmetric spaces due to Delorme and S\'echerre \cite{delorme--secherre2011}. In particular, there exists a compact subset $\style C$ of $G$ and a finite subset $\style X$ of $(\Hbf \mathbf M_0)(F)$ such that $G = \style C S_0^+\style X^{-1} H$, where $\style X^{-1} = \{x^{-1} : x \in \style X\}$.
If $0< \epsilon_1 \leq \epsilon_2 \leq 1$, then $S_\Theta^-(\epsilon_1) \subset S_\Theta^-(\epsilon_2) \subset S_\Theta^-$.  By definition,
$S_\Theta^+(\epsilon) = \{ s \in S_\Theta : s^{-1} \in S_\Theta^-(\epsilon)\}$; therefore $S_\Theta^+(\epsilon_1) \subset S_\Theta^+(\epsilon_2) \subset S_\Theta^+$.  In particular, as $\epsilon$ goes to zero the cosets $\style C S_\Theta^+(\epsilon) x^{-1} H$ shrink, and $\style C S_\Theta^+(\epsilon_1) x^{-1} H \subset \style C S_\Theta^+(\epsilon_2) x^{-1} H$, if $0< \epsilon_1 \leq \epsilon_2 \leq 1$. We have the following.

\begin{cor}\label{support-not-RSC}
Let $\chi$ be an arbitrary quasi-character of $H$ and 
let $\lambda \in \Hom_H(\pi,\chi)$ be a nonzero $(H,\chi)$-equivariant form on an admissible $(H,\chi)$-distinguished representation $(\pi,V)$ of $G$. 
Let $S$ be the $(\theta,F)$-split component of a proper $\theta$-split parabolic $P$ of $G$.  If there exists $v\in V$ such that the function $\varphi_{\lambda,v}$ is non-zero on a sequence $\{s_n\}_{n\in \N}$ approaching $S$-infinity, then $\pi$ is not $(H,\chi,\lambda)$-relatively supercuspidal.
\end{cor}

\begin{proof}
Apply \Cref{lem-twisted-kt-6-7-contrapositive} to show that $\lambda_{N,\chi} \neq 0$, then the result follows from \Cref{thm-twisted-rsc-condition}.
\end{proof}

\subsection{The proof of \Cref{thm-twisted-not-rsc}}
We work under the assumptions of \Cref{thm-twisted-not-rsc} and we prove the contrapositive statement:

\begin{thm}[Contrapositive of \Cref{thm-twisted-not-rsc}]\label{thm-twisted-not-rsc-contrapositive}
Let $\chi$ be an unramified quasi-character of $H$.
If $\delta_Q^{1/2}\rho$ is not $(L^\theta, \delta_{Q^\theta}\chi\vert_{L^\theta}, \lambda)$-relatively supercuspidal, then the induced representation $\pi=\iota_Q^G\rho$ is not $(H,\chi,\lambda^G)$-relatively supercuspidal.
\end{thm}

Let $S$ be a maximal non-central $(\theta,F)$-split torus of $L$.\footnote{By assumption, $L$ admits an irreducible non-$(L^\theta, \delta_{Q^\theta}\chi\vert_{L^\theta}, \lambda)$-supercuspidal representation; therefore, $Z_LL^\theta\backslash L$ must be non-compact.  If the $(\theta,F)$-split component $S_L$ of $L$ is a maximal $(\theta,F)$-split torus in $L$, then $C_L(S_L)=L$ and by \cite[Proposition 4.7]{helminck--wang1993} $L$ has  no proper $\theta$-split parabolic subgroups.  Moreover, if this is case, then $S_L L^\theta \backslash L$, and hence $Z_L L^\theta \backslash L$, is compact by the Relative Cartan Decomposition \cite{delorme--secherre2011}.} 
By assumption $\delta_Q^{1/2}\rho$ is not $(L^\theta,\delta_{Q^\theta}\chi\vert_{L^\theta},\lambda)$-relatively supercuspidal; therefore, there is some relative matrix coefficient $\varphi_{\lambda, v}$ of $\delta_Q^{1/2}\rho$ that is not compactly supported modulo $S_L L^\theta$. 
By the Relative Cartan Decomposition \cite{delorme--secherre2011} applied to $L$, there is a compact subset $\style C_L$ of $L$ and a finite subset $\style X_L$ of 
$(\mathbf{L}^\theta C_\mathbf{L}(S))(F)$ such that $L = \style C_L S^+ \style X_L^{-1} L^\theta$.
The support of $\varphi_{\lambda, v}$ is not contained in any subset of $L$ of the form $C S_L L^\theta$, where $C\subset L$ is compact.
In fact, since $\style X_L$ is finite and 
\begin{align*}
\varphi_{\lambda,v}(\ell'\ell) &= \delta_{Q^\theta}(\ell)^{-1}\chi(\ell)^{-1}\varphi_{\lambda,v}(\ell'), & \text{for all}\ \ell' \in L, \ell \in L^\theta,
\end{align*}
there exists a sequence $\{\ell_n\}_{n\in \N} \subset L$ such that $\varphi_{\lambda,v}(\ell_n) \neq 0$, and where $\ell_n = o_n s_n x^{-1}$, where $o_n \in \style C_L$, $x\in \style X_L$ (a fixed element), and $\{s_n\} \subset S^+$ approaches $S$-infinity.\footnote{This relies on the fact that $\delta_{Q^\theta}(\ell)^{-1}\chi(\ell)^{-1}$ is nonzero for all $\ell \in L^\theta$.}
Replacing $S$ with the maximal $(\theta,F)$-split torus $S' = x S x^{-1}$ of $L$, we have that $\varphi_{\lambda,v}$ is nonzero on the sequence $\{\ell_n = o_n's_n'\}_{n\in \N}$, with $o_n' = o_n x^{-1}$,  $s_n' = x s_n x^{-1}$, and where $\{s_n'\}$ approaches $S'$-infinity.
The sequence $\{\ell_n\}$ has non-compact image modulo $S_L L^\theta$ and modulo $S_G H$ (recall $S_G \subset S_L$ and $L^\theta \subset H$).  

The idea of the proof of \Cref{thm-twisted-not-rsc} is to understand the support of the $\lambda^G$-relative matrix coefficients of $\pi$ in terms of the support of the $\lambda$-relative matrix coefficients of $\delta_{Q}^{1/2}\rho$.\footnote{We can consider only $\rho$ since $\rho$ and $\delta_Q^{1/2}\rho$ act on the same vector space; moreover, $\Hom_{L^\theta}(\delta_Q^{1/2}\rho,\delta_{Q^\theta}\chi\vert_{L^\theta})$ can be identified with $\Hom_{L^\theta}(\rho,\delta_{Q^\theta}(\delta_Q^{-1/2}\chi)\vert_{L^\theta})$.}
For simplicity, we assume that we can find a vector $v$ in the space  $V_\rho \cong V_{\delta_Q^{1/2}\rho}$ of $\rho$ such that $\varphi_{\lambda,v}$ is not compactly supported on $S/S_L$.
If this is not the case, we may still produce $\varphi_{\lambda,v}$ and a sequence $\{\ell_n : \varphi_{\lambda,v}(\ell_n) \neq 0\}$ with non-compact image modulo $S_L L^\theta$ (and modulo $S_G H$), as above;
the argument below to produce a non-compactly supported $\lambda^G$-relative matrix coefficient $\varphi_{\lambda^G, f_v}$ of $\pi$ still goes through with only a slight adjustment.  In particular, one need only apply \Cref{lem-twisted-kt-6-7-contrapositive} instead of \Cref{support-not-RSC}.
\begin{proof}[Proof of \Cref{thm-not-rsc}]
Let $V_\rho$ be the space of $\rho$ and identify $V_\rho$ with the space of $\delta_Q^{1/2}\rho$.
Let $\pi = \iota_Q^G\rho$ and let $V_\pi$ be the space of $\pi$.
Let $S$ be a maximal $(\theta,F)$-split torus $S$ of $L$.  As above, note that $S$ is non-central in $L$.  
We show that if $\delta_Q^{1/2}\rho$ admits a non-compactly supported (modulo $Z_LL^\theta$) $\lambda$-relative matrix coefficient, then $\pi$ admits a non-compactly supported (modulo $Z_G H$) $\lambda^G$-relative matrix coefficient.
In fact, our goal is to construct a vector $f\in V_\pi$ such that the $\lambda^G$-relative matrix coefficient $\varphi_{\lambda^G,f}$ is non-compactly supported modulo $Z_G H$.
First, we will study an arbitrary  $\lambda^G$-relative matrix coefficient $\varphi_{\lambda^G,f}$.  
Let $f\in V_\pi$ and consider the $\lambda^G$-relative matrix coefficient $\varphi_{\lambda^G,f} \in C_{\omega}^\infty(G,H,\chi)$.
Using the definition of $\lambda^G$ given in \eqref{def-lambdaG}, we see that for any $\ell \in L$
\begin{align*}
\varphi_{\lambda^G,f}(\ell) & = \ip{\lambda^G}{\pi(\ell)^{-1}f} \\
 & =  \int_{Q^\theta \backslash H} \ip{\lambda}{\chi(h)^{-1}[\pi(\ell)^{-1}f](h)} \ d\mu(h) \\
 &=  \int_{Q^\theta \backslash H} \ip{\lambda}{\chi(h)^{-1} f (h\ell^{-1})} \ d\mu(h) \\
 &=  \int_{Q^\theta \backslash H} \ip{\lambda}{\chi(h)^{-1} f (\ell^{-1}\ell h\ell^{-1})} \ d\mu(h) \\
 &= \int_{Q^\theta \backslash H} \ip{\lambda}{\chi(h)^{-1} \delta_Q^{1/2}(\ell^{-1})\rho(\ell^{-1}) f(\ell h \ell^{-1})} \ d\mu(h). 
\end{align*}
That is, for any element $\ell$ of $L$, we have
\begin{align}\label{int-lambda-rel}
\varphi_{\lambda^G,f}(\ell) 
 &= \int_{Q^\theta \backslash H} \ip{\lambda}{\delta_Q^{1/2}(\ell^{-1})\rho(\ell^{-1}) \chi(h)^{-1} f(\ell h \ell^{-1})} \ d\mu(h).
\end{align}
If we can appropriately control the elements $\chi(h)^{-1}f(\ell h \ell^{-1}) \in V_\rho$, then the integrand in \eqref{int-lambda-rel}  is essentially a $\lambda$-relative matrix coefficient of $\delta_Q^{1/2}\rho$.  
We will link the support of the $\lambda$-relative matrix coefficients of $\delta_Q^{1/2}\rho$ with the support of the $\lambda^G$-relative matrix coefficients of $\pi$, via the integral \eqref{int-lambda-rel}.

Let $K$ be a compact open subgroup of $G$ that has Iwahori factorization with respect to $Q = LU$.  Then the product map (in any order)
\begin{align*}
(L\cap K) \times (U\cap K) \times (U^{\op} \cap K) \rightarrow K,
\end{align*}
is bijective.  
Take a vector $v\in V_\rho$ such that the map given by
\begin{align*}
	s\mapsto \ip{\lambda}{\delta_Q^{1/2}(s^{-1})\rho(s^{-1})v} &= \delta_Q^{1/2}(s^{-1})\varphi_{\lambda,v}(s), & s \in S
\end{align*}
is not compactly supported on $S$ modulo $S_L$.  
Define $f_v \in V_\pi$ to be zero off of $QK = Q(U^{\op} \cap K)$ 
and such that $f_v(\bar u) = v$ for $\bar u \in U^{\op}\cap K$.  Requiring $f_v \in V_\pi$ completely determines the function $f_v: G \rightarrow V_\rho$.  Now consider the possible values of $f_v(\ell h \ell^{-1})$. By construction, $f_v$ is zero unless $\ell h \ell^{-1} \in Q(U^{\op} \cap K)$.
Suppose that $\ell h \ell^{-1} \in Q(U^{\op} \cap K)$. This occurs if and only if $h$ lies in
\begin{align}\label{ch6-equation}
(\ell^{-1}Q \ell) \ell^{-1} (U^{\op} \cap K) \ell & = Q(U^{\op} \cap \ell^{-1} K \ell),
\end{align}
where the equality in \eqref{ch6-equation} holds since $\ell \in L$ normalizes $U$ and $U^{\op}$.  Write $h = q\bar u$ where $q \in Q$ and $\bar u \in U^{\op} \cap \ell^{-1} K \ell$.  Since $h$ is $\theta$-fixed,
\begin{align*}
q\bar u = h = \theta(h) = \theta(q) \theta(\bar u),
\end{align*}
where $\theta(q) \in Q$ and $\theta(\bar u) \in U^{\op}$ since both subgroups are $\theta$-stable.  The product map $(q',\bar u')  \mapsto q'\bar u'$ on
\begin{align*}
L\times U \times U^{\op} = Q\times U^{\op} & \rightarrow G
\end{align*}
is one-to-one; therefore, $\theta(q) = q$ and $\theta(\bar u ) = \bar u$.
It follows that $h$ is equivalent to $\bar u$ in $Q^\theta \backslash H$; moreover,
$ \ell \bar u \ell^{-1} \in U^{\op} \cap K$ since $\bar u \in (U^{\op} \cap \ell^{-1} K \ell)^\theta$ and so $f_v(\ell \bar u \ell^{-1})= v$.

The compact open subgroup $\ell^{-1}K\ell$ also has Iwahori factorization with respect to $Q$. In particular, the image of $(\ell^{-1}K\ell)^\theta$ in $Q^\theta \backslash H$ is equal to the image of $(U^{\op} \cap \ell^{-1} K \ell)^\theta$ which is open  
and thus has positive measure.
From \eqref{int-lambda-rel}, integrating over the image of $(U^{\op} \cap \ell^{-1} K \ell)^\theta$ in $Q^\theta \backslash H$, we obtain 
\begin{align*}
\varphi_{\lambda^G,f_v}(\ell) & = \ip{\lambda^G}{\pi(\ell^{-1})f_v}\\
& = \int_{Q^\theta \backslash H} \ip{\lambda}{\delta_Q^{1/2}(\ell^{-1})\rho(\ell^{-1}) \chi(h)^{-1} f_v(\ell h \ell^{-1})} \ d\mu(h)\\
& = \int_{Q^\theta \backslash (U^{\op} \cap \ell^{-1} K \ell)^\theta} \ip{\lambda}{\delta_Q^{1/2}(\ell^{-1})\rho(\ell^{-1}) \chi(\bar{u})^{-1} f_v(\ell \bar{u} \ell^{-1})} \ d\mu(\bar{u})\\
& = \int_{Q^\theta \backslash (U^{\op} \cap \ell^{-1} K \ell)^\theta} \chi(\bar{u})^{-1}\delta_Q^{1/2}(\ell^{-1})\ip{\lambda}{\rho(\ell^{-1})  v} \ d\mu(\bar{u})\\
& = \delta_Q^{1/2}(\ell^{-1}) \varphi_{\lambda,v}(\ell)   \int_{Q^\theta \backslash (U^{\op} \cap \ell^{-1} K \ell)^\theta} \chi(\bar{u})^{-1} \ d\mu(\bar{u}) \\
& = c_{\ell,\chi} \cdot \delta_Q^{1/2}(\ell^{-1}) \varphi_{\lambda,v}(\ell),
\end{align*}
where 
\begin{align}\label{eq-volume-const}
	 c_{\ell,\chi} & = \int_{Q^\theta \backslash (U^{\op} \cap \ell^{-1} K \ell)^\theta} \chi(\bar{u})^{-1} \ d\mu(\bar{u}).
\end{align}
By assumption, $\chi$ is unramified\footnote{Ultimately, one would like to remove this assumption on $\chi$; however, we need to guarantee that $c_{\ell,\chi}\neq 0$ for $\ell \in S \subset L$ approaching $S$-infinity.}; therefore, $\chi$ is trivial on the compact unipotent group $(U^{\op} \cap \ell^{-1} K \ell)^\theta$.Thus, we have that $c_{\ell,\chi}$ is nonzero and equal to the volume $\mu(Q^\theta\backslash(U^{\op} \cap \ell^{-1} K \ell)^\theta) > 0$ of $Q^\theta\backslash(U^{\op} \cap \ell^{-1} K \ell)^\theta$.
It is immediate that if the support of $\varphi_{\lambda, v}$ on $S/S_L$ is non-compact then the same holds for $\varphi_{\lambda^G, f_v}$.
In particular, by the assumption on $\varphi_{\lambda, v}$, there exists a sequence $\{s_n\}_{n\in \N}$ approaching $S$-infinity such that $\varphi_{\lambda, v}(s_n) \neq 0$, for all $n \in \N$.
For this same sequence, we have that 
$\varphi_{\lambda^G, f_v}(s_n) \neq 0$, for every $n \in \N$.
By \Cref{support-not-RSC}, $\pi$ is not $(H,\chi,\lambda^G)$-relatively supercuspidal.
\end{proof}

\subsection*{Acknowledgements}
The author would like to thank the Automorphic Representations Research Group at the University of Calgary for several helpful discussions.
Thank you to Fiona Murnaghan for suggesting this project.  Thank you to Omer Offen for his encouragement and many helpful suggestions.  Finally, thank you to Shuichiro Takeda for informing the author of his independent work on \Cref{thm-twisted-rsc-condition} and \Cref{thm-twisted-kt08-7-1}.

\bibliographystyle{amsalpha}

\bibliography{matrix-coeff}

\providecommand{\bysame}{\leavevmode\hbox to3em{\hrulefill}\thinspace}
\providecommand{\MR}{\relax\ifhmode\unskip\space\fi MR }
\providecommand{\MRhref}[2]{%
  \href{http://www.ams.org/mathscinet-getitem?mr=#1}{#2}
}
\providecommand{\href}[2]{#2}
\begin{thebibliography}{Smi18b}

\bibitem[BD08]{blanc--delorme2008}
Philippe Blanc and Patrick Delorme, \emph{Vecteurs distributions
  {$H$}-invariants de repr{\'e}sentations induites, pour un espace
  sym{\'e}trique r{\'e}ductif {$p$}-adique {$G/H$}}, Ann. Inst. Fourier
  (Grenoble) \textbf{58} (2008), no.~1, 213--261. \MR{2401221 (2009e:22015)}

\bibitem[BZ76]{bernstein--zelevinsky1976}
I.~N. Bernstein and A.~V. Zelevinsky, \emph{Representations of the group
  {$GL(n,F),$} where {$F$} is a local non-{A}rchimedean field}, Uspehi Mat.
  Nauk \textbf{31} (1976), no.~3(189), 5--70. \MR{0425030 (54 \#12988)}

\bibitem[Cas95]{Casselman-book}
William Casselman, \emph{Introduction to the theory of admissible
  representations of $p$-adic reductive groups}, Unpublished manuscript, draft
  prepared by the S\'{e}minaire Paul Sally, 1995.

\bibitem[CD14]{carmona--delorme2014}
Jacques Carmona and Patrick Delorme, \emph{Constant term of {E}isenstein
  integrals on a reductive {$p$}-adic symmetric space}, Trans. Amer. Math. Soc.
  \textbf{366} (2014), no.~10, 5323--5377. \MR{3240926}

\bibitem[Del10]{delorme2010}
Patrick Delorme, \emph{Constant term of smooth {$H_\psi$}-spherical functions
  on a reductive {$p$}-adic group}, Trans. Amer. Math. Soc. \textbf{362}
  (2010), no.~2, 933--955. \MR{2551511}

\bibitem[DS11]{delorme--secherre2011}
Patrick Delorme and Vincent S{{\'e}}cherre, \emph{An analogue of the {C}artan
  decomposition for {$p$}-adic symmetric spaces of split {$p$}-adic reductive
  groups}, Pacific J. Math. \textbf{251} (2011), no.~1, 1--21. \MR{2794612
  (2012f:22011)}

\bibitem[GO16]{gurevich--offen2016}
Maxim Gurevich and Omer Offen, \emph{A criterion for integrability of matrix
  coefficients with respect to a symmetric space}, J. Funct. Anal. \textbf{270}
  (2016), no.~12, 4478--4512. \MR{3490774}

\bibitem[Hel91]{helminck1991a}
Aloysius~G. Helminck, \emph{Tori invariant under an involutorial automorphism.
  {I}}, Adv. Math. \textbf{85} (1991), no.~1, 1--38. \MR{1087795}

\bibitem[HH98]{helminck--helminck1998}
A.~G. Helminck and G.~F. Helminck, \emph{A class of parabolic {$k$}-subgroups
  associated with symmetric {$k$}-varieties}, Trans. Amer. Math. Soc.
  \textbf{350} (1998), no.~11, 4669--4691. \MR{1443876 (99g:20082)}

\bibitem[HW93]{helminck--wang1993}
A.~G. Helminck and S.~P. Wang, \emph{On rationality properties of involutions
  of reductive groups}, Adv. Math. \textbf{99} (1993), no.~1, 26--96.
  \MR{1215304 (94d:20051)}

\bibitem[KT08]{kato--takano2008}
Shin-Ichi Kato and Keiji Takano, \emph{Subrepresentation theorem for {$p$}-adic
  symmetric spaces}, Int. Math. Res. Not. IMRN (2008), no.~11. \MR{2428854}

\bibitem[KT10]{kato--takano2010}
\bysame, \emph{Square integrability of representations on {$p$}-adic symmetric
  spaces}, J. Funct. Anal. \textbf{258} (2010), no.~5, 1427--1451. \MR{2566307}

\bibitem[Lag08]{lagier2008}
Nathalie Lagier, \emph{Terme constant de fonctions sur un espace sym{\'e}trique
  r{\'e}ductif {$p$}-adique}, J. Funct. Anal. \textbf{254} (2008), no.~4,
  1088--1145. \MR{2381204 (2009d:22013)}

\bibitem[Off06]{offen2006b}
Omer Offen, \emph{Residual spectrum of {${\rm GL}_{2n}$} distinguished by the
  symplectic group}, Duke Math. J. \textbf{134} (2006), no.~2, 313--357.
  \MR{2248833 (2007h:11063)}

\bibitem[Off17]{offen2017}
\bysame, \emph{On parabolic induction associated with a {$p$}-adic symmetric
  space}, J. Number Theory \textbf{170} (2017), 211--227. \MR{3541705}

\bibitem[Smi17]{smith-phd2017}
Jerrod~Manford Smith, \emph{Construction of relative discrete series
  representations for $p$-adic $\mathbf{GL}_n$}, Ph.D. thesis, University of
  Toronto, June 2017.

\bibitem[Smi18a]{smith2018b}
Jerrod~Manford Smith, \emph{Local unitary periods and relative discrete
  series}, Pacific J. Math. \textbf{297} (2018), no.~1, 225--256. \MR{3864235}

\bibitem[Smi18b]{smith2018}
\bysame, \emph{Relative discrete series representations for two quotients of
  {$p$}-adic {$\bold{GL}_n$}}, Canad. J. Math. \textbf{70} (2018), no.~6,
  1339--1372. \MR{3850546}

\end{thebibliography}
%

\end{document}